\documentclass[12pt,english]{article}
\usepackage[T1]{fontenc}
\usepackage[latin9]{inputenc}
\usepackage{geometry}
\usepackage{xcolor}
\usepackage{amsmath}
\usepackage{amsthm}
\usepackage{amssymb}
\usepackage{hyperref}
\usepackage{authblk}

\makeatletter

\providecommand{\tabularnewline}{\\}

\theoremstyle{plain}
\newtheorem{thm}{\protect\theoremname}
\theoremstyle{remark}
\newtheorem{rem}[thm]{\protect\remarkname}
\theoremstyle{plain}
\newtheorem{lem}[thm]{\protect\lemmaname}

\@ifundefined{date}{}{\date{}}
\usepackage{babel}
\providecommand{\lemmaname}{Lemma}
\providecommand{\remarkname}{Remark}
\providecommand{\theoremname}{Theorem}

\makeatother

\usepackage{babel}
\providecommand{\lemmaname}{Lemma}
\providecommand{\remarkname}{Remark}
\providecommand{\theoremname}{Theorem}

\usepackage{natbib}
\begin{document}
\baselineskip .32in

\title{Inbound Replenishment and Outbound Dispatch Decisions under Hybrid Shipment Consolidation Policies: An Analytical Model and Comparison}
%Comparison of Integrated Inventory Replenishment and Outbound Dispatch Models under Shipment Consolidation: A Martingale Approach

%\author{Bo Wei \\ weibomath@gmail.com
 %  \and S{\i}la \c{C}etinkaya \\ sila@lyle.smu.edu
  %  \and Daren B.H. Cline \\ dcline@stat.tamu.edu
   %}
%\author[$a$]{Bo Wei}
%\author[$b$]{S{\i}la \c{C}etinkaya}
%\author[$c$]{Daren B.H. Cline}
%
%\affil[$a$]{IORA, National University of Singapore, Singapore}
%\affil[$b$]{Engineering Management, Information, and Systems, Southern Methodist University, Dallas, USA}
%\affil[$c$]{Department of Statistics, Texas A\&M University, College Station, USA}

\author{Bo Wei \qquad S{\i}la \c{C}etinkaya\qquad Daren B.H. Cline}

\maketitle
\begin{abstract}
\noindent We consider a distribution warehouse where both the inbound inventory replenishment and outbound dispatch decisions are subject to fixed (as well as per-unit) transportation charges and demand is stochastic. In order to realize scale economies associated with transportation operations both on the outbound and inbound sides, dispatch schedules must be synchronized over time with inventory replenishment decisions.  An immediate delivery policy on the outbound side does not make economic sense because outbound dispatch operations and costs will benefit from a \emph{temporal shipment consolidation policy}. Our particular interest in this setting is the exact modeling and analysis of \emph{hybrid} shipment consolidation policies, in comparison to the  time- and quantity-based counterparts.
%To this end, we propose exact and approximate analytical methods to compute and then compare the cost and delivery performance as well as inventory stagnancy/flow under hybrid policies relative to its alternatives.
Since shipment consolidation prolongs customer waiting and inventory holding, we investigate average delay penalty
%(penalties including both linear and squared delay)
per order and average inventory per time unit as critical measures of  performance of the distribution operation, along with the annual cost. By fixing the expected inbound replenishment and outbound dispatch frequencies, we compare these measures among the alternative ways of operation under hybrid, time-based, and quantity-based policies. This comparison then lends itself to an explicit analytical comparison of average costs under these three policies, without a need for solving the corresponding optimization problems. \\

%Notably, our results offer a full analytical characterization of relative cost performance and demonstrates the impact and value of alternative shipment consolidation policies regardless of the values of model parameters. The results are of practical value in the context of  design and operation of an integrated framework for inventory-transportation systems.\\

\end{abstract}

%\newpage

%\tableofcontents

%\newpage

\section{Introduction and Motivations}

A growing body of academic publications, along with a particular industry focus, have placed an emphasis on the value of coordination of different \emph{functional specialties} (e.g., finance, marketing, production/inventory, logistics/transportation, etc.)  in the context of supply chain management practices (\cite{axsa15,BGA14,CL00,Marklund11,SKS00,SKMK16,SMA18}).
Notably, coordination of production/inventory and logistics/transportation operations has received a significant attention over the past two decades as more companies have implemented specific programs enabling such coordination, including but not limited to vendor-managed inventory (VMI), third-party logistics (3PL), and time definite delivery (TDD) practices  (\cite{AYK12,CETI04,CML06,CTL08,CL00,GAB14,JCL07,KKO13_IJPE,KKO13_CIE,LCJ03,PCL20}).

The existing body of work supporting the implementation of VMI, 3PL, TDD and similar programs can be categorized as deterministic (e.g., see \cite{CL02,CUEK09,hwang10,JCL07,KKO13_IJPE,LCJ03}) and stochastic (e.g., see \cite{Axsa01,CML06,CTL08,CL00,KCT14,KKO13_CIE,Marklund11,MC10,SKMK16}) demand models where several practical generalizations related to the latter category remain as open and challenging research problems. More specifically,
despite a notable body of literature on the integration of inventory and transportation decisions, there is still a need for exact analytical decision making models considering stochastic demand. In this paper, we revisit a stylized, yet, fundamental, problem setting with the goal of addressing this need and filling a void in the existing literature.

In particular, we consider a stock-keeping point, e.g., a distribution warehouse, where both the inbound inventory replenishment and outbound dispatch decisions are subject to fixed (as well as per-unit) transportation charges. These charges are associated with lump-sum transportation (e.g., set-up, loading, bundling, equipment, etc.) costs of each inbound and outbound delivery, regardless of quantity delivered to satisfy stochastic demands realized over time. Hence, outbound dispatch decisions must be synchronized with inbound inventory replenishment decisions over time so as to realize scale economies associated with logistical operations both on the inbound and outbound sides.

Unlike the traditional approaches to inventory control, an immediate delivery policy on the outbound side does not make economic sense in the context of the problem at hand. This is because of the fixed costs of outbound dispatch operations which will benefit from the implementation of a \emph{temporal shipment consolidation policy}. Such a policy is characterized by active efforts to merge several smaller stochastic demands realized over time into larger dispatch quantities to be shipped as combined loads realizing scale economies (\cite{CB03,HB94,HB95}).

Three kinds of temporal shipment consolidation policies are common in the logistics practice and previous literature (e.g., see \cite{CHB14,CETI04,CB03,CML06,CMW14,CTL08,CL00,HBC20,HB94,HB95,MCB10,SEB18,WCC20}), and they are revisited in this paper:
 (i) Quantity-based policy (QP) under which stochastic demands are held until a target dispatch quantity $q$ is accumulated;
(ii) Time-based policy (TP) under which a consolidated dispatch quantity is released by a pre-determined time (say, periodically, every $T$ time-units; and
(iii) Hybrid policy (HP) or time-and-quantity policy under which the shipper attempts to accumulate a target
dispatch quantity $q$ but releases the existing consolidated load by a predetermined time (say, at most after $T$ time-units since the last dispatch) if the size of the consolidated load has not reached $q$ in a timely fashion.
Clearly, QPs assure transportation scale economies are fully realized whereas TPs ensure timely delivery. HPs are aimed at balancing the economic benefits versus timely delivery trade-offs associated with QPs and TPs (\cite{CETI04,CML06,CMW14,HB94,MCB10,WCC20}). For the case where demand is a Poisson process, as in this paper, by assumption, the policy parameters are such that $T > 0$ and $q$ is a positive integer.

As noted in \cite{CETI04}, temporal shipment consolidation may be implemented on its own without
coordination, a.k.a., a {\em pure} consolidation practice. However, for all practical purposes, it is imperative to consider the impact of temporal shipment consolidation on inventory and dispatch decisions throughout the supply chain. Thus, a more informed approach targeting supply chain coordination is
to integrate/synchronize temporal shipment consolidation with production/inventory and logistics/transportation decisions with a holistic vision. This latter approach is referred to as an
{\em integrated} inventory/consolidation practice (\cite{CETI04}).

Our focus in this paper is the exact modeling of HPs in an integrated  framework for making inbound replenishment and outbound dispatch decisions simultaneously. An accompanying goal is to offer an analytical comparison of the resulting model with the previously developed counterpart models that consider TPs (\cite{CL00}) and QPs (\cite{CML06}). That is, the overall analysis presented here is based on three alternative models of operation to include the integrated inbound replenishment and outbound dispatch models under QPs, TPs, and HPs.

More specifically, we propose exact and approximate analytical methods to compute and then compare the (i) cost performance, (ii) delivery performance, and (iii) inventory stagnancy/flow under HPs relative to its alternatives, i.e., QPs and TPs. Clearly, all three types of shipment consolidation policies are implemented at the expense of prolonged customer waiting and inventory holding before an outbound dispatch is made to satisfy stochastic demands realized over time. Hence, in addition to annual cost, we investigate average delay penalties (including both linear and squared delay) per order and average inventory per time unit as additional critical measures of interest associated with the performance of distribution operation under investigation. By fixing the expected replenishment cycle length (i.e., inbound replenishment frequency) and/or the expected shipment consolidation cycle length (i.e., outbound dispatch frequency),
%(accordingly, the expected inventory replenishment cost per time unit and the expected dispatch cost per time unit are fixed),
we are able to compare these two measures among the three alternative models of operation.

This comparison then lends itself to an explicit analytical comparison of average cost criteria of the three operational models, without a need for solving the corresponding three different optimization problems. Notably, our comparative results lead to a full analytical characterization of cost performance under the alternative models of operation with shipment consolidation. Thus, we fill a gap in literature via a set of analytically provable results which were only illustrated by numerical tests previously in \cite{CML06}. Our analytically provable results justify the impact and value of alternative shipment consolidation policies regardless of the values of model parameters and are of practical value in the context of  design and operation of an integrated framework for inventory-transportation systems.

The remainder of this paper is organized as follows. An overview of previously
established results and our specific contributions relative to the existing literature are summarized in Section \ref{LiterandContri}.
In Section \ref{sec:IM_HP},
we introduce the underlying stochastic processes and operational system characteristics of the integrated model under HP, and then we provide an exact formulation and analysis of the resulting model. Multiple critical performance measures applicable in the context of
the three alternative models of operation (to include the integrated inventory replenishment and outbound dispatch models under QPs, TPs, and HPs) are proposed and the main
comparison results in terms of these measures along with long-run average cost are given in Section \ref{sec:Comparison-of-service}.
Finally, Section \ref{sec:Conclusion} concludes the study with a
discussion of further generalizations and related directions for future research.

\section{Previous Literature and Our Contributions}
\label{LiterandContri}

Research in pure consolidation has been inspirational and instrumental for the development and implementation of
integrated practices and models. Hence, we begin with a review of the previous quantitative modeling efforts in this area in Section \ref{pure} which is followed by a review of the relevant work on integrated consolidation practices in Section \ref{integrated}. An overview of our results and the accompanying practical insights are then discussed in Section \ref{overview}.

\subsection{Related Literature on Pure Consolidation Practices}
\label{pure}

A growing body of literature has focused on development of quantitative models supporting pure consolidation practices since mid 90s. Notably, three related streams of research contribute to this literature: (i) numerical and analytical models aimed at computation of operational policy parameters for implementing pure shipment consolidation practices in industrial logistics (\cite{CHB14,CB03,CMW14,BH02,HB95,HBC20})
and disaster relief (\cite{CL17,ZSHZ19}) applications;
(ii) admission control, cost allocation, and pricing models for managing pure shipment consolidation practices (\cite{SEB18,UB12,WCC20}); and
(iii) performance analysis of pure policies for improving shipment consolidation operations (\cite{HB94,MCB10,WCC20}).
Among these three, the results offered by the last stream of research is the most relevant for our purposes, and, thus, summarized here.

A detailed simulation study in \cite{HB94} reveals that
that QPs achieve the least cost compared with the other policies.
However, in terms of average waiting time, HPs outperform both TPs and QPs. Later, the analytical and numerical analyses  in \cite{CMW14} provide a comparison of the distribution of maximum waiting time ($MWT$) and the average order delay ($AOD$) among QPs, TPS, and HPs, where demand is modeled as a Poisson process and
\begin{equation}
AOD=\frac{\mathbb{E}[\mbox{Cumulative waiting per consolidation cycle}]}{\mathbb{E}[\mbox{Number of orders arriving in a consolidation cycle}]}.
\label{aod-1}
\end{equation}
Here, the term delay refers to the time between arrival and clearing of a realized (i.e., outstanding) demand due to outbound dispatch delay associated with shipment consolidation policy in place. Hence, both $MWT$ and $AOD$ are treated as indicators of \emph{timely service} (i.e., service performance). Under fixed policy parameters, $q$ and/or $T$, the analytical results in \cite{CMW14} indicate that HPs outperform QPs and TPs, not only in terms of the $P(MWT>t)$ for a given $t$; but, also, in terms of $AOD$. Also, the numerical results in \cite{CMW14} demonstrate that for a fixed expected
consolidation cycle length (i.e., dispatch frequency), while QPs perform the best in terms of $AOD$, the general class of HPs performs better than the general class of counterpart TPs. More recently, the full set of analytical results in \cite{WCC20} provide a general analytical characterization of $AOD$,
leading to a complete comparative analysis of alternative renewal-type clearing policies including but not limited to QPs, TPs, and HPs. That is, regardless of the parametric setting at hand, for a fixed expected
consolidation cycle length, QPs are the best, and HPs outperform counterpart TPs in terms of $AOD$. Hence, the analysis in \cite{WCC20} closes the gaps in literature regarding service performance of pure shipment consolidation practices via a complete set of analytically provable results regarding $AOD$ which were only illustrated through
numerical tests in \cite{CMW14}.
%Moreover, it is analytically shown that for a fixed expected consolidation cycle length, the HP with larger quantity parameter would achieve larger $AOD$ than the HP with smaller quantity parameter.

In light of these recent analytical findings illuminating the clear advantages of HPs over TPs, along with the cost versus service trade-offs between QPs and HPs, in the context of pure consolidation practices, there is value in comparing these policies in an integrated framework. Such a comparison happens to be the core objective of the current paper. In order to support this objective, we now discuss additional results from the relevant literature on integrated consolidation practices.

\subsection{Related Literature on Integrated Consolidation Practices}
\label{integrated}

The existing work on integrated consolidation practices was inspired by the problem setting in \cite{CL00} and focuses on computation of policy parameters for simultaneous optimization of inventory replenishment (\cite{Axsa01,CML06,CTL08,CL00,MC10}) or production (\cite{KKO13_IJPE,KKO13_CIE}) quantities along with outbound dispatch decisions to minimize the expected long run average cost per time unit.
Hence, we also consider the problem setting of \cite{CL00} where the problem was motivated in the context of a VMI application in which the upper echelon is a vendor that is
(i) operating the distribution warehouse;
(ii) implementing the TP in order to serve a group of downstream retailers of a market area with Poisson demands; and (iii) aiming to optimally schedule the upstream replenishment quantities and the downstream consolidated shipments.

In traditional inventory control applications, demands are assumed to be satisfied immediately; but, under the stipulations of a VMI program, the vendor has the authority to combine and hold on to the delivery of small orders from the market area for a reasonable amount of time in order to achieve transportation scale economies. This mode of operation leads to a basic push-pull system where the initial stage of the supply chain (vendor) is operated in a push-based
manner while the remaining stages (retailers in the market area) employ a pull-based strategy (\cite{SKS00}) taking advantage of the potential benefits of temporal shipment consolidation practices.

The analysis in \cite{CL00} focuses on the implementation of a TP on the outbound side and relies on an approximate expression of the cost function. The problem setting of \cite{CL00} have been revisited in \cite{CML06} in order to compare the impact of HPs with QPs and TPs, again under Poisson demand. However, the comparative results presented therein also rely on approximate cost expression of \cite{CL00} along with a simulation study without any analytical results applicable to HPs. Two significant generalizations of the problem to consider general renewal demand processes and common-carrier costs have been studied in  \cite{CTL08} and \cite{MC10}, respectively; again, developing approximate expressions of cost functions and without any consideration of HPs. Hence, our focus here is the exact modeling and analysis under a HP in an integrated  framework for making inbound replenishment and outbound dispatch decisions simultaneously. The resulting model is referred to as the integrated model under HP in the remainder of the paper and has been compared with the integrated models under TP (\cite{CL00}) and QP (\cite{CML06}). Our comparative results include both cost-based and service-based indicators of performance for all three models of interest as explained in the next section.

According to the detailed numerical study of \cite{CML06} with approximate and simulation results,
%that focuses on the three integrated models of interest here,
HPs remain advantageous in the context of integrated consolidation practices, too, because they are cost-wise superior to TPs and service-wise superior to QPs. By providing an exact analytical model and analysis of the integrated consolidation practice under HPs, we aim to validate and offer an analytical justification for these observations revealed only numerically in \cite{CML06}. Also, in \cite{CML06}, the service aspect is measured by the long run average cumulative waiting time only
whereas we consider more specific performance indicators, as explained in the next section, leading to a more comprehensive set of comparative results.

\subsection{Overview of Our Results and Accompanying Practical Insights}
\label{overview}

As we have noted earlier, shipment consolidation policies are in place at the expense of both prolonged customer waiting and inventory holding. Hence, in addition to studying the exact cost expressions of integrated models under TPs, QPs, and HPs, we also (i) investigate average delay penalties per order (through $AOD$ as well as $AOSD$ as defined below) and (ii) average inventory per time unit (through $AIR$ as defined below), as critical measures of performance of the integrated distribution operations. More specifically, we are interested in a comparative analysis of the following criteria.
\begin{itemize}
    \item \textbf{Cost:} All three integrated models of interest in this paper take into account for four cost components; namely, inventory replenishment, inventory carrying, dispatch, and customer waiting costs. Our comparative results rely on the computation of exact expressions of these cost components for obtaining the exact expression of expected long run average cost per time unit.
    \item \textbf{Order delay:} We consider two measures in order to quantify order delay.
    \begin{itemize}
    \item[1.] In the spirit of previous work on pure consolidation practices (\cite{CMW14, WCC20}), we consider $AOD$, defined in (\ref{aod-1}), as an indicator of service performance of integrated consolidation practices, too.
    \item[2.] In $AOD$, given by Eq. (\ref{aod-1}), the waiting penalty is assumed to be linear to the delivery delay. However, in some situations, due to the impatience of the customer, it also makes sense to consider the case where the waiting penalty is proportional to the square of the delivery delay encountered due to shipment consolidation practice. The corresponding service criterion we propose for this purpose is then given by
\begin{equation}
AOSD=\frac{\mathbb{E}[\mbox{Cumulative squared waiting per consolidation cycle}]}{\mathbb{E}[\mbox{Number of orders arriving in a consolidation cycle}]}.
\label{aosd-1}
\end{equation}
    \end{itemize}
    \item \textbf{Inventory stagnancy/flow:} Last but not least, we investigate average inventory per time unit, i.e., average inventory rate, denoted by $AIR$ and given by
\begin{equation}
AIR=\frac{\mathbb{E}[\mbox{Cumulative inventory carried per replenishment cycle}]}{\mathbb{E}[\mbox{Replenishment cycle length}]},
\label{air-1}
\end{equation}
as an additional measure of performance.
\end{itemize}

As we have noted earlier, our comparative investigation relies on the computation of exact expressions of the components of expected long run average cost per time unit along with
the exact expressions of $AOD$, $AOSD$, and $AIR$. These expressions include terms involving truncated random variables, and their development rely on the use of a combination of renewal theory, martingales, and stochastic calculus. Our analysis of the resulting expressions also leads to the discovery of new characteristics of truncated random variables involved.  Throughout this paper, for a non-negative integer-valued random variable $X$ and a positive integer $q$,
$
X_q\triangleq\min(X, q)
$
denotes the truncated random variable of interest.

%In the integrated model under HP, to compute the exact expected inventory carrying cost per replenishment cycle (see Section \ref{expected inventory holding}), we apply a renewal argument to derive a renewal type equation. The solution of the renewal type equation first leads to the exact expression of the expected cumulative inventory carried, and, then, the exact cost term of interest. To compute the exact expected customer waiting cost per shipment consolidation cycle under HP (see Section \ref{expected cumulative linear delay}), we utilize a martingale point of view and rely on the aid of the martingale stopping theorem. In order to compute the expected cumulative squared delay per shipment consolidation cycle under HP (see Section \ref{expected squared delay} and Appendix \ref{sec:Squared-Delay}), we rely on stochastic calculus.

 The main contributions of our work can be summarized as follows:
\begin{enumerate}
\item By applying a renewal theoretical framework and utilizing martingales and stochastic calculus, we offer an exact formulation of the integrated model under HP, which was studied via simulation only in the existing literature.
\item By developing and using new and refined properties of truncated Poisson random variables, we are able to show that for a fixed expected consolidation cycle length under the integrated models of interest,
\begin{itemize}
    \item QP performs the best and TP performs the
worst in terms of $AOD$, while HP lies between QP and TP;
\item both QP and HP lead to smaller $AOSD$ than TP; but,
\item QP does not necessarily achieve the smallest $AOSD$, which
is different from the corresponding result in terms of $AOD$
(see Theorems \ref{thm:AOD} and \ref{thm:AOSD}).
\end{itemize}
\item Moreover, relying on a reasonable approximation technique,
%similar to the one in \cite{CL00}
we are able to show that for a fixed expected replenishment cycle length and a fixed expected consolidation cycle length under the integrated practice of interest,
QP has smaller $AIR$ than both HP and TP in an approximate fashion, and $AIR$ under HP is approximately the same as that under TP (see Theorem \ref{thm:AIR comparison}).
%\item Hence, for a fixed expected shipment consolidation cycle length,  we offer a complete analytical comparison of the three integrated models of operation  under QPs, TPs, and HPs in terms of $AOD$ and $AOSD$. Moreover, through an approximation, we offer an analytical comparison of these models in terms of $AIR$ under a fixed expected consolidation length and a fixed replenishment cycle length.
\item Building on the analytical comparison results in terms of $AOD/AOSD$ and $AIR$, we are also able to compare the expected long run average cost per time unit for the three integrated models of operation regardless of the parametric setting, without a need to solve the corresponding optimization problems explicitly and, hence, eliminating the burden of a numerical solution approach completely.

\item Considering the critical measures of performance, including cost-based and service-based criteria, we present  insightful guidelines informing the design and operation of integrated inventory-transportation systems.
\end{enumerate}

\section{\label{sec:IM_HP}Integrated Model under HP}
We begin with a more detailed description of the integrated problem setting under HP and introduce the underlying stochastic processes of interest in Section \ref{sec:model_description}. Then we proceed with the model formulation in Section \ref{sec:Analysis of IM_HP}. Throughout the remainder of this paper, for two real numbers $a$ and $b$, let $a\wedge b\triangleq\min(a,b)$.

\subsection{\label{sec:model_description} Problem Setting and Underlying Stochastic Processes}

We consider a distribution warehouse (e.g., vendor) serving a market area, potentially consisting of multiple customers (e.g., retailers) with stochastic demands realized over time. The demand process of each customer is assumed to follow a Poisson process so that the aggregate demand process of the market area, denoted by $N(t)$, is also Poisson, say with rate $\lambda$.
The inventory at the distribution warehouse is replenished instantaneously by incurring a fixed replenishment cost of $A_R$ and a per unit replenishment cost of $c_R$. Any inventory held at the distribution warehouse is subject to an inventory holding cost of $h$ per unit per time unit. The outbound dispatch decisions for transporting goods to satisfy the stochastic demand of market area are also subject to a fixed dispatch cost of $A_D$ along with a per unit dispatch cost of $C_D$ at the expense of the distribution warehouse.  Hence, outbound dispatch schedules must be synchronized over time with inbound inventory replenishment decisions, and outbound dispatch operations and costs take advantage of an explicit shipment consolidation practice under a HP. Since shipment consolidation is implemented at the expense of prolonged customer waiting, outstanding orders waiting to be delivered to the market are subject to a waiting cost penalty of $\omega$ per unit per time unit.

Let $q_{H}$ and $T_{H}$ denote the parameters associated with the HP of interest, under which a dispatch decision
is made every $\tau_{q_{H}}\wedge T_{H}$
time units, where $\tau_{q_{H}}$ is first hitting time of $q_{H}$
with respect to the aggregate Poisson demand process.
By assumption, the policy parameters are such that $T_H > 0$ and $q_H$ is a positive integer.

A shipment consolidation
cycle is then defined as the time between
two successive outbound dispatch decisions. Since the distribution warehouse is able to replenish its inventory in bulk immediately from an external source with plentiful supply, it suffices to review the inventory level at the end of each shipment consolidation cycle. At the review instant, if the on-hand inventory is not sufficient to clear all outstanding orders, the distribution center first replenishes its
stock and then releases the entire consolidated load such that
the remaining on-hand inventory is a positive integer, say $Q_{H}$.

An inventory
replenishment cycle is then defined as the time between two
successive inventory replenishment decisions. Thus, the decision variables are $q_{H}$ and $T_{H}$ which trigger a dispatch decision, and the order-up-to level $Q_{H}$ which indicates the base-stock quantity.

Let $L(t)$ and $I(t)$ denote the size of the consolidated load
waiting to be released and the inventory level at the distribution warehouse at time $t$, respectively. Then, at the end of a shipment consolidation cycle, we have one of the two cases:
\begin{itemize}
\item
%If the amount of on-hand inventory is not enough to clear the entire consolidated load at this time, i.e.,
If $I(t)<L(t)$ at the end of a consolidation cycle then a replenishment quantity of size $Q_{H}+L(t)-I(t)$ is ordered and received immediately. Next, the entire consolidated load is dispatched immediately so that a new replenishment cycle (also a new consolidation cycle) starts
with $Q_{H}$ units of on hand inventory.
\item %If the amount of on-hand inventory is sufficient to deliver the entire consolidated load at this time, i.e.,
If $I(t)\geq L(t)$ at the end of a consolidation cycle, then the consolidated load is dispatched using
the on-hand inventory. The replenishment cycle
continues but a new shipment consolidation cycle starts with $I(t)-L(t)$ units of on-hand inventory.
\end{itemize}

That is, it suffices for the distribution warehouse to employ an $(s,S)$ policy with $s=-1$ and $S=Q_{H}$ for replenishing the inventory. After dispatching the consolidated demand  of the previous consolidation cycle, the inventory on hand is always between $0$ and $Q_{H}$. Thus,  the inventory process $I(t)$ is a regenerative process. The regeneration
points are the epochs at which the vendor's inventory is replenished (the inventory level is renewed to be $Q_{H}$) so that $I(t)$ consists of $i.i.d.$ replenishment cycles.

We let $L_{n}^{R}$ denote the length of the $n$-th replenishment cycle
and $TCost_{n}$ denote the total cost incurred during the $n$-th
replenishment cycle, where $n=1,2,\ldots$. Then, the pairs $\left(L_{n}^{R},\,TCost_{n}\right)$,
$n\geq1$, are $i.i.d.$. Since the order inter-arrival times have
a finite mean, and the costs are assumed to be finite, $\mathbb{E}\left[L_{n}^{R}\right]=\mathbb{E}\left[L_{HP}^{R}\right]<\infty$
and $\mathbb{E}\left[TCost_{n}\right]=\mathbb{E}\left[TCost_{HP}\right]<\infty$,
$n\geq1$. It then follows from the renewal reward theorem (\cite{Ross96}) that the \emph{expected long run average cost per time unit} associated with the operation of the distribution warehouse, denoted by $AC_{HP}(q_{H},T_{H},Q_{H})$  can be computed by using
\begin{equation}
AC_{HP}(q_{H},T_{H},Q_{H})=\frac{\mathbb{E}[TCost_{HP}]}{\mathbb{E}[L_{HP}^{R}]}.\label{eq:average cost rate}
\end{equation}
The optimization problem at hand is to minimize this cost function in order to compute the optimal values of $q_{H},T_{H}$ and $Q_{H}$.

\subsection{\label{sec:Analysis of IM_HP}Model Formulation: Fundamental Expressions of Interest}

Next, we proceed with deriving an explicit expression
of $AC_{HP}(q_{H},T_{H},Q_{H})$ by calculating the expected replenishment
cycle length $\mathbb{E}[L_{HP}^{R}]$ and expected cost of a replenishment
cycle $\mathbb{E}[TCost_{HP}]$.

Clearly, the consolidation
process $L(t)$ is also a regenerative process where successive outbound dispatches represent regeneration epochs. Let $L_{i}^{C}$ denote the
length of the $i$-th consolidation cycle (i.e., time between two
dispatch decisions), and let $N_{i}$ denote the size of the consolidated
load accumulated during the $i$-th consolidation cycle. In our setting, $\{L_{i}^{C}\}_{i\geq 1}$
and $\{N_{i}\}_{i\geq 1}$ are two sequences of $i.i.d.$
random variables. It follows that for all $i\geq1$, $L_{i}^{C}$ has the same probability distribution as $\tau_{q_{H}}\wedge T_{H}$.
For all $i\geq1$,  $N_{i}$ has the same probability distribution as $N(\tau_{q_{H}}\wedge T_{H})$. Note that $N(\tau_{q_{H}}\wedge T_{H})$ can be rewritten as a truncated random variable $Y_{q_{H}}=\min\left(Y,q_{H}\right)$,
where $Y\sim\text{Poisson}(\lambda T_{H})$.

Observe that $N(t)-\lambda t$ is a martingale
and $\tau_{q_{H}}\wedge T_{H}$ is a bounded stopping time. Using
the martingale stopping theorem in \cite[p.300]{Ross96}, we have $\mathbb{E}\left[N(\tau_{q_{H}}\wedge T_{H})\right]=\lambda\mathbb{E}\left[\tau_{q_{H}}\wedge T_{H}\right]$.
Thus, the expected consolidation cycle length is
\begin{equation}
\mathbb{E}\left[L_{i}^{C}\right]=\mathbb{E}\left[L_{HP}^{C}\right]=\mathbb{E}[\tau_{q_{H}}\wedge T_{H}]=\frac{1}{\lambda}\mathbb{E}\left[N(\tau_{q_{H}}\wedge T_{H})\right]=\frac{1}{\lambda}\mathbb{E}\left[Y_{q_{H}}\right],\label{eq:consolidation cycle}
\end{equation}
and the expected number of orders arriving in one consolidation cycle is
\begin{eqnarray}
\mathbb{E}\left[N_{i}\right]=\mathbb{E}\left[Y_{q_{H}}\right]. \label{eq:consolidation size}
\end{eqnarray}

\subsubsection{Expected Replenishment Cycle Length}

Each inventory replenishment cycle consists of at least
one shipment consolidation cycle. Let $K_{H}$ denote the number of shipment
consolidation cycles within a replenishment cycle, which is
defined as
\begin{eqnarray}
\label{KH}
K_{H}\triangleq\min\left\{k\ \text{is\ a\ positive\ integer}:\sum_{j=1}^{k}N_{j}>Q_{H}\right\}.
\end{eqnarray}
The length of an inventory replenishment cycle in the integrated model under HP with parameters $q_{H}$ and $T_{H}$ is $L_{HP}^{R}=\sum_{j=1}^{K_{H}}L_{j}^{C}$.
We observe that $\sum_{j=1}^{K_{H}}L_{j}^{C}$ is a finite stopping
time with respect to $N(t)$, and for any $t>0$,
\[
\Big|N\left(\sum_{j=1}^{K_{H}}L_{j}^{C}\wedge t\right)-\lambda\sum_{j=1}^{K_{H}}L_{j}^{C}\wedge t\Big|\leq q_{H}K_{H}+\lambda T_{H}K_{H},
\]
where the random variable in the right hand side is integrable. Thus,
by Proposition 2.5.7(iii) in \cite[p.65]{AL06}, $\left\{ N\left(\sum_{j=1}^{K_{H}}L_{j}^{C}\wedge t\right)-\lambda\sum_{j=1}^{K_{H}}L_{j}^{C}\wedge t\right\} _{t\geq0}$
is a uniformly integrable martingale. Moreover, by the martingale
stopping theorem and Vitali convergence theorem (the uniform integrability convergence theorem) in \cite[p.105]{Rosenthal2006}, we have
\[
\mathbb{E}\left[N\left(\sum_{j=1}^{K_{H}}L_{j}^{C}\right)\right]=\lambda\mathbb{E}\left[\sum_{j=1}^{K_{H}}L_{j}^{C}\right].
\]
It then follows that
\begin{equation}
\mathbb{E}\left[L_{HP}^{R}\right]=\mathbb{E}\left[\sum_{j=1}^{K_{H}}L_{j}^{C}\right]=\frac{1}{\lambda}\mathbb{E}\left[N\left(\sum_{j=1}^{K_{H}}L_{j}^{C}\right)\right]=\frac{1}{\lambda}\mathbb{E}\left[\sum_{j=1}^{K_{H}}N_{j}\right]=\frac{1}{\lambda}\mathbb{E}\left[K_{H}\right]\mathbb{E}\left[Y_{q_{H}}\right],\label{eq:replenishment cycle}
\end{equation}
where the last equality follows from the Wald's equation (\cite{Ross96}) because
$K_{H}$ is a stopping time with respect to the sequence $\{N_{j}\}_{j\geq 1}$. With this observation in mind, we compute $\mathbb{E}[K_{H}]$ next.

Using the definition of $K_{H}$, we then have $\left\{K_{H}\geq k\right\}\Leftrightarrow\left\{\sum_{j=1}^{k-1}N_{j}\leq Q_{H}\right\}$
for all $k\in\left\{ 1,2,\ldots\right\} $. Let $G_{H}(\cdot)$ as
the distribution function of $Y_{q_{H}}$ and $G_{H}^{(k)}(\cdot)$
as the $k$-fold convolution of $G_{H}(\cdot)$. Moreover, we define
$M_{G_{H}}(i)\triangleq\sum_{k=0}^{\infty}G_{H}^{(k)}(i)$. Then we
have $\mathbb{P}(K_{H}\geq k)=G_{H}^{(k-1)}(Q_{H})$ for all $k\in\left\{ 1,2,\ldots\right\} $
and
\begin{equation}
\mathbb{E}[K_{H}]=\sum_{k=1}^{\infty}\mathbb{P}(K_{H}\geq k)=\sum_{k=1}^{\infty}G_{H}^{(k-1)}(Q_{H})=1+\sum_{k=1}^{\infty}G_{H}^{(k)}(Q_{H})=M_{G_{H}}(Q_{H}).\label{eq:EK}
\end{equation}
We substitute Eq.~\eqref{eq:EK} into Eq.~\eqref{eq:replenishment cycle}
and obtain
\begin{equation}
\mathbb{E}\left[L_{HP}^{R}\right]=\frac{1}{\lambda}\mathbb{E}\left[Y_{q_{H}}\right]M_{G_{H}}(Q_{H}).\label{eq:expected replenish cycle length}
\end{equation}

\subsubsection{Expected Inventory Replenishment Cost per Replenishment cycle}

Let $\mathbb{E}[RCost_{HP}]$ denote the replenishment cost per replenishment
cycle. Since $K_{H}$ is a stopping time with respect to the sequence $\{N_{j}\}_{j\geq 1}$,
by the Wald's equation, we have
\[
\mathbb{E}[\text{Replenishment\ \ Quantity}]=\mathbb{E}\left[\sum_{j=1}^{K_{H}}N_{j}\right]=\mathbb{E}[K_{H}]\mathbb{E}\left[Y_{q_{H}}\right],
\]
so that
\begin{align}
\mathbb{E}[RCost_{HP}] & =A_{R}+c_{R}\mathbb{E}[\text{Replenishment\ \ Quantity}] =A_{R}+c_{R}\mathbb{E}[K_{H}]\mathbb{E}\left[Y_{q_{H}}\right].\label{eq:replenishment cost}
\end{align}

\subsubsection{Expected Inventory Carrying Cost per Replenishment Cycle and $AIR$}\label{expected inventory holding}

Under the HP, the inventory dynamics
within a replenishment cycle is as follows,
\[
I(t)=\begin{cases}
Q_{H}, & 0\leq t\leq L_{1}^{C},\\
Q_{H}-N_{1}, & L_{1}^{C}<t\leq\sum_{j=1}^{2}L_{j}^{C},\\
\vdots\\
Q_{H}-\sum_{j=1}^{K_{H}-1}N_{j}, & \sum_{j=1}^{K_{H}-1}L_{j}^{C}<t\leq\sum_{j=1}^{K_{H}}L_{j}^{C}.
\end{cases}
\]
Letting  $\mathbb{E}[H_{HP}]$ denote the expected cumulative inventory carrier per replenishment cycle, we have
\[
\mathbb{E}\left[H_{HP}\right]\triangleq H(Q_{H},q_{H},T_{H})=\mathbb{E}\left[\int_{0}^{\sum_{j=1}^{K_{H}}L_{j}^{C}}I(t)dt\right].
\]
Using the renewal argument,
\begin{align*}
H(Q_{H},q_{H},T_{H}|N_{1}=i) & =\begin{cases}
\mathbb{E}\left[\tau_{q_{H}}\wedge T_{H}\right]Q_{H}, & \mbox{if}\quad i\geq Q_{H}+1,\\
\mathbb{E}\left[\tau_{q_{H}}\wedge T_{H}\right]Q_{H}+H(Q_{H}-i,q_{H},T_{H}), & \mbox{if}\quad i\leq Q_{H},
\end{cases}
\end{align*}
so that
\[
H(Q_{H},q_{H},T_{H})=\mathbb{E}\left[\tau_{q_{H}}\wedge T_{H}\right]Q_{H}+\sum_{i=0}^{Q_{H}}H(Q_{H}-i,q_{H},T_{H})g_{H}(i),
\]
where $g_{H}(\cdot)$ denotes the probability mass function of $Y_{q_{H}}$.
We further denote $g_{H}^{(k)}(\cdot)$ as the $k$-fold convolution
of $g_{H}(\cdot)$ and define $m_{g_{H}}(i)\triangleq\sum_{k=0}^{\infty}g_{H}^{(k)}(i)$.
The expression for $H(Q_{H},q_{H},T_{H})$ is a renewal type equation,
its solution is given as
\begin{align}
H(Q_{H},q_{H},T_{H}) & =\mathbb{E}\left[\tau_{q_{H}}\wedge T_{H}\right]\sum_{i=0}^{Q_{H}}(Q_{H}-i)m_{g_{H}}(i) =\frac{1}{\lambda}\mathbb{E}[Y_{q_{H}}]\sum_{i=0}^{Q_{H}}(Q_{H}-i)m_{g_{H}}(i).\label{eq:holding}
\end{align}

\begin{rem}
Note that for $i=0,1,\ldots,Q_{H}$, $m_{g_{H}}(i)$ is the sum of
the probabilities to reach inventory $Q_{H}-i$ after $k=0,1,2,\ldots$
dispatches. In the integrated model under HP (see \cite{Axsa01,axsa15,ZF91} for details),
\[
\frac{m_{g_{H}}(i)}{M_{G_{H}}(Q_{T})}=\frac{m_{g_{H}}(i)}{\sum_{i=0}^{Q_{H}}m_{g_{H}}(i)}\]
is the steady state probability of inventory level $Q_{H}-i$, where
$i=0,1,\ldots,Q_{H}.$
\end{rem}

Now, letting $\mathbb{E}[HCost_{HP}]$ denote the expected inventory holding
cost within one replenishment cycle under HP, we have
\begin{align}
\mathbb{E}[HCost_{HP}] & =hH(Q_{H},q_{H},T_{H}) =\frac{h}{\lambda}\mathbb{E}[Y_{q_{H}}]\sum_{i=0}^{Q_{H}}(Q_{H}-i)m_{g_{H}}(i).\label{eq:inventory cost}
\end{align}
Also, substituting Eqs.~\eqref{eq:expected replenish cycle length} and \eqref{eq:holding} in Eq.~(\ref{air-1}), it then follows that the $AIR$ associated with the integrated model under HP is given by
\begin{eqnarray}
AIR_{HP}=\frac{\sum_{i=0}^{Q_{H}}(Q_{H}-i)m_{g_{H}}(i)}{M_{G_{H}}(Q_{H})}.\label{AIR_HP}
\end{eqnarray}

Based on a visual inspection of the right hand side of Eq.~\eqref{AIR_HP}, it is easy to conclude that $AIR_{HP}$ is difficult to evaluate both analytically and numerically. This difficulty is rooted at the
same issue with the form of Eq.~\eqref{eq:expected replenish cycle length} which is used for obtaining Eq.~\eqref{AIR_HP}. Hence, the approximations presented below are useful proxies when an exact comparative analysis is not possible as in the case of evaluating $AIR$ expressions (e.g., see Section \ref{results}).

From the definition of $K_{H}$ in Eq.~\eqref{KH},
\[
\mathbb{E}\left[\sum_{j=1}^{K_{H}}N_{j}\right]  =\mathbb{E}[K_{H}]\mathbb{E}[Y_{q_{H}}]\geq Q_{H}+1
\quad \mbox{and} \quad
\mathbb{E}\left[\sum_{j=1}^{K_{H}-1}N_{j}\right]  =\mathbb{E}[K_{H}]\mathbb{E}[Y_{q_{H}}]-\mathbb{E}[Y_{q_{H}}]\leq Q_{H}.
\]
As result,
\[\frac{Q_{H}}{\mathbb{E}[Y_{q_{H}}]}+1\geq\mathbb{E}[K_{H}]\geq\frac{Q_{H}}{\mathbb{E}[Y_{q_{H}}]}+\frac{1}{\mathbb{E}[Y_{q_{H}}]}.
\]
Then, treating $K_{H}$ as a continuous random variable, %(this approximation technique is also used in \cite{CL00}, where the integrated model under TP is studied),
we have
\begin{eqnarray*}
\mathbb{E}\left[\sum_{j=1}^{K_{H}}N_{j}\right]=\mathbb{E}\left[K_{H}\right]\mathbb{E}\left[Y_{q_{H}}\right]\approx Q_{H}+1,
\end{eqnarray*}
which, in turn, implies that
\begin{eqnarray}
\mathbb{E}[K_{H}]\approx\frac{Q_{H}+1}{\mathbb{E}[Y_{q_{H}}]}.\label{appox of EK}
\end{eqnarray}
It then follows from from Eq.~(\ref{eq:replenishment cycle}) that
\begin{eqnarray}
\mathbb{E}\left[L_{HP}^{R}\right]=\frac{1}{\lambda}\mathbb{E}[K_{H}]\mathbb{E}\left[Y_{q_{H}}\right]\approx\frac{Q_{H}+1}{\lambda}.\label{appr expected repl length}
\end{eqnarray}
Moreover, using Eq.~(\ref{eq:EK}), we also have
\begin{eqnarray*}
M_{G_H}(Q_{H})\approx\frac{Q_{H}+1}{\mathbb{E}[Y_{q_{H}}]},
\end{eqnarray*}
and
\begin{eqnarray*}
m_{g_H}(i)=M_{G_H}(i)-M_{G_H}(i-1)\approx\frac{1}{\mathbb{E}[Y_{q_{H}}]},\label{mg}
\end{eqnarray*}
so that Eq.~(\ref{eq:holding}) leads to
\begin{eqnarray}
\mathbb{E}[H_{HP}]\approx\frac{1}{\lambda}\mathbb{E}[Y_{q_{H}}]Q_{H}+\frac{1}{2\lambda}(Q_{H}+1)Q_{H}.\label{approxholding}
\end{eqnarray}
Substituting Eqs.~\eqref{appr expected repl length} and \eqref{approxholding} in Eq.~(\ref{air-1}), it then follows that
\begin{eqnarray}
AIR_{HP}\approx\frac{Q_{H}\left(2\mathbb{E}[Y_{q_{H}}]+Q_{H}+1\right)}{2(Q_{H}+1)}.\label{apprAIR_HP}
\end{eqnarray}

\subsubsection{Expected Dispatch Cost per Replenishment Cycle}

Let $\mathbb{E}[DCost_{HP}]$ denote the dispatch cost per replenishment
cycle. Since all outstanding demands are dispatched every $\tau_{q_{H}}\wedge T_{H}$
units of time and $K_{H}$ is the number of shipment consolidation
cycles within one replenishment cycle, we have
\begin{align}
\mathbb{E}\left[DCost_{HP}\right] & =A_{D}\mathbb{E}[K_{H}]+c_{D}\mathbb{E}\Big[\sum_{j=1}^{K_{H}}N_{j}\Big] =A_{D}\mathbb{E}[K_{H}]+c_{D}\mathbb{E}[K_{H}]\mathbb{E}[Y_{q_{H}}].\label{eq:dispatch cost}
\end{align}

\subsubsection{Expected Waiting Cost per Replenishment Cycle, $AOD$, and $AOSD$}\label{expected cumulative linear delay}

In the spirit of previous work in shipment consolidation, we associate a linear, time-dependent  penalty, denoted by $\omega$, associated with customer waiting. Since the shipment consolidation cycle length under a HP with parameters
$q_{H}$ and $T_{H}$ is $\tau_{q_{H}}\wedge T_{H}$ and the corresponding cumulative
linear delay within one shipment consolidation cycle is $W_{HP}=\int_{0}^{\tau_{q_{H}}\wedge T_{H}}N(t)dt$,
the expected cumulative waiting within one shipment
consolidation cycle is
\[
\mathbb{E}\left[W_{HP}\right]=\mathbb{E}\left[\int_{0}^{\tau_{q_{H}}\wedge T_{H}}N(t)dt\right].
\]

While the results in \cite{MCB10} already provide a useful, yet, involved method to compute the expected cumulative waiting within one consolidation cycle under HP in the context of pure consolidation practices, here we propose a drastically simplified approach based on a martingale argument associated with the Poisson demand process.

\begin{lem}
\label{lem:martin} Let $N(t)$ be a Poisson process with rate
$\lambda$, and define \textup{$W(t)\triangleq\int_{0}^{t}N(u)du$.}
Then
\[
W(t)-\frac{1}{2\lambda}N^{2}(t)+\frac{1}{2\lambda}N(t)
\]
is a martingale with respect to the natural filtration \textup{$\{\mathcal{G}_{t}\}$}, which is the $\sigma$-field generated by the family of random variables $\{N(s), s\in [0,t]\}$.
\end{lem}

\begin{proof}
See Appendix \ref{sec:Proofs} for all proofs.
\end{proof}

Using Lemma \ref{lem:martin}, noting $\tau_{q_{H}}\wedge T_{H}$
is a bounded stopping time with respect to the demand process $N(t)$, and applying the martingale stopping theorem, we obtain
\begin{eqnarray}
 \mathbb{E}[W_{HP}]=\mathbb{E}[W(\tau_{q_{H}}\wedge T_{H})] =\frac{1}{2\lambda}\mathbb{E}[N^{2}(\tau_{q_{H}}\wedge T_{H})-N(\tau_{q_{H}}\wedge T_{H})] =\frac{1}{2\lambda}\mathbb{E}[Y_{q_{H}}(Y_{q_{H}}-1)].\label{linear delay}
\end{eqnarray}
Now, letting $\mathbb{E}[WCost_{HP}]$ denote the expected waiting cost per replenishment
cycle and recalling that $K_{H}$ is the number of shipment consolidation cycles
within one replenishment cycle, we have
\begin{align}
\mathbb{E}[WCost_{HP}] & =\omega\mathbb{E}[K_{H}]\mathbb{E}[W_{HP}]=\frac{\omega}{2\lambda}\mathbb{E}[K_{H}]\mathbb{E}[Y_{q_{H}}(Y_{q_{H}}-1)].\label{eq:waiting cost}
\end{align}
Also, substituting Eqs.~\eqref{eq:consolidation size} and \eqref{linear delay} in Eq.~\eqref{aod-1}, it then follows that the $AOD$ associated with the integrated model under HP is given by
\begin{eqnarray}
AOD_{HP}=\frac{\mathbb{E}[Y_{q_{H}}(Y_{q_{H}}-1)]/(2\lambda)}{\mathbb{E}[Y_{q_{H}}]}.\label{AOD_HP}
\end{eqnarray}

While the cost penalty considered in $\mathbb{E}[WCost_{HP}]$ is assumed to be linear to the delivery delay, in practice, due to customer impatience, this assumption may not be realistic. Hence, we also consider the case
where this penalty may be proportional to the square of the delivery delay encountered by the customer. The counterpart expected cumulative
squared waiting within one shipment consolidation cycle is then
\[
\mathbb{E}\left[W'_{HP}\right]=\mathbb{E}\left[\int_{0}^{\tau_{q_{H}}\wedge T_{H}}(\tau_{q_{H}}\wedge T_{H}-t)^{2}dN(t)\right],
\]
and as we show in Appendix \ref{sec:Squared-Delay},
\begin{eqnarray}
\mathbb{E}\left[W'_{HP}\right]=\frac{1}{3\lambda^{2}}\mathbb{E}[Y_{q_{H}+1}(Y_{q_{H}+1}-1)(Y_{q_{H}+1}-2)].\label{squareddelay}
\end{eqnarray}

Noting that HP with parameters $q$ and
$T$ reduces
\begin{itemize}
    \item  to QP with parameter $q$ when $T\rightarrow\infty$, and
    \item to TP with parameter $T$ when $q\rightarrow\infty$,
\end{itemize}
and using the index TP or QP to represent the consolidation policy type to generalize the notation we have adopted so far, we also have the expected cumulative
squared waiting within one consolidation cycle under QP with parameter $q$ is
\begin{eqnarray}
\mathbb{E}\left[W'_{QP}\right]=\lim_{T\rightarrow\infty}\mathbb{E}[W'_{HP}]=\frac{q^{3}-q}{3\lambda^{2}},
\label{squareddelay-QP}
\end{eqnarray}
and the expected cumulative squared waiting within one consolidation cycle under TP with parameter $T$ is
\begin{eqnarray}
\mathbb{E}\left[W'_{TP}\right]=\lim_{q\rightarrow\infty}\mathbb{E}\left[W'_{HP}\right]=\frac{\mathbb{E}[Y(Y-1)(Y-2)]}{3\lambda^{2}}=\frac{\lambda T^{3}}{3}.\label{squared delay-TP}
\end{eqnarray}

Clearly, one can then associate a cost penalty, say ${\omega}^{'}$, with ${E}[W'_{.}]$ leading to alternative expressions of the expected long run average cost functions of the three integrated models. Our primary purpose here in developing the above expressions of ${E}[W'_{.}]$ is to utilize them in examining the counterpart $AOSD$ expressions as indicators of service performance with impatient customers which have not been investigated in the previous literature. For this purpose, substituting Eqs.~\eqref{eq:consolidation size} and \eqref{squareddelay} in Eq.~(\ref{aosd-1}), we have the following $AOSD$ expression associated with the integrated model under HP:
\begin{eqnarray}
    AOSD_{HP}= \frac{\mathbb{E}\left[Y_{q_{H}+1}(Y_{q_{H}+1}-1)(Y_{q_{H}+1}-2)\right]/(3\lambda^{2})}{\mathbb{E}[Y_{q_{H}}]}. \label{AOSD_HP}
\end{eqnarray}

Likewise, using Eqs.~\eqref{squareddelay-QP} and \eqref{squared delay-TP} and noting that expected number of orders arriving in one consolidation cycle under QP with parameter $q$ and TP with parameter $T$ is $q$ and $\lambda T$, respectively, it can be easily verified that the $AOSD$ expressions associated with the integrated models under QP and TP are:
\begin{eqnarray}
    AOSD_{QP}=\frac{q^{2}-1}{3\lambda^{2}} \quad \mbox{and} \quad  AOSD_{TP}= T^{2}/3.\label{AOSD_QP and TP}
\end{eqnarray}
Since $AOSD$ has not been studied in \cite{CML06} and \cite{CL00}, the counterpart expressions are developed here.

\subsubsection{Exact Expression of the Cost Function under HP}

Last but not least, using Eqs.~(\ref{eq:replenishment cost}), (\ref{eq:inventory cost}),
(\ref{eq:dispatch cost}), and (\ref{eq:waiting cost}),
in $
\mathbb{E}[TCost_{HP}]=\mathbb{E}[RCost_{HP}]+\mathbb{E}[HCost_{HP}]+\mathbb{E}[DCost_{HP}]+\mathbb{E}[WCost_{HP}]$, and recalling Eqs.~(\ref{eq:average cost rate}) and (\ref{eq:expected replenish cycle length}), we have
\begin{eqnarray}
AC_{HP}(Q_{H},q_{H},T_{H})
& = &  \lambda(c_{R}+c_{D})+ \frac{\lambda A_{R}}{M_{G}(Q_{H})\mathbb{E}[Y_{q_{H}}]}+
\frac{\lambda A_{D}}{\mathbb{E}[Y_{q_{H}}]}
\nonumber \\
&& +\frac{h\sum_{i=0}^{Q_{H}}(Q_{H}-i)m_{g_{H}}(i)}{M_{G}(Q_{H})}
+\frac{\omega\mathbb{E}\left[Y_{q_{H}}(Y_{q_{H}}-1)\right]}{2\mathbb{E}[Y_{q_{H}}]}.\label{eq:objectiveofHP}
\end{eqnarray}

\section{\label{sec:Comparison-of-service}Comparative Analysis}

We now set the stage for a detailed comparative analysis of the operation of the distribution warehouse considering the integrated models under QP (as presented in \cite{CML06}), TP ( as presented in \cite{CL00}), and HP (as developed in the previous section). Since our comparative analysis draws from \cite{CML06} and \cite{CL00}, first in Section \ref{summary} we recall the relevant results therein directly without derivations and provide a summary of the key expressions we use for comparisons. Because the subsequent analysis relies on some new and refined properties of truncated Poisson random variables, next in Section \ref{trun-prop}, we discuss such relevant formal results. Finally, in Section \ref{results}, we compare $AOD$, $AOSD$ and $AIR$ as critical measures of interest of the three integrated models of investigation, and then we focus on a comparison of these models in terms of the expected long run average cost per time unit. Interestingly, we are able to offer a complete analytical comparison of the cost criteria without a need for explicitly solving the three underlying optimization problems.

\subsection{Summary of Expressions Utilized for Comparisons}
\label{summary}

In this subsection, we directly cite the relevant results from \cite{CML06} and \cite{CL00} for the integrated models under QP and TP, respectively. We adopt a self-explanatory and mnemonic notation using the policy identifier (QP, TP, HP) as an index.

We let $T$ and $Q_{T}$ denote the consolidation cycle length and the order-up-to level in the integrated model under TP for which $K_{T}$ is the number of consolidation cycles within one replenishment cycle.
%Also, we let $g_{T}(\cdot)$ and $G_{T}(\cdot)$ denote the probability mass function and the distribution function of $N(T)$ (a Poisson random variable with mean $\lambda T$), respectively, and we define $g_{T}^{(k)}(\cdot)$ and $G_{T}^{(k)}(\cdot)$ as the $k$-fold convolution of $g_{T}(\cdot)$ and $G_{T}(\cdot)$, respectively, that is, $g_{T}^{(k)}(i)=(k\lambda T)^{i}e^{-k\lambda T}/i!$ and $G_{T}^{(k)}(j)=\sum_{i=0}^{j}(k\lambda T)^{i}e^{-k\lambda T}/i!$. Moreover, we let $m_{g_{T}}(i)\triangleq\sum_{k=0}^{\infty}g_{T}^{(k)}(i)$ and $M_{G_{T}}(i)\triangleq\sum_{k=0}^{\infty}G_{T}^{(k)}(i)$.
From the results in
\cite{CL00}, $\mathbb{E}[K_{T}]=M_{G_{T}}(Q_{T})$, where
\[
M_{G_{T}}(i)\triangleq\sum_{k=0}^{\infty}G_{T}^{(k)}(i), \quad
g_{T}^{(k)}(i)=\frac{(k\lambda T)^{i}e^{-k\lambda T}}{i!}, \; \mbox{and} \; G_{T}^{(k)}(j)=\sum_{i=0}^{j}g_{T}^{(k)}(i).
\]
Also, the expected
cumulative inventory carrying per replenishment cycle under TP is
$\mathbb{E}[H_{TP}]=T\sum_{i=0}^{Q_{T}}(Q_{T}-i)m_{g_{T}}(i)$, where $m_{g_{T}}(i)\triangleq\sum_{k=0}^{\infty}g_{T}^{(k)}(i)$, and the counterpart expected cumulative waiting time per consolidation cycle is $\mathbb{E}[W_{TP}]=\lambda T^{2}/2$. Since $\mathbb{E}[K_{T}]$ cannot be computed in closed form, using a technique similar to the one used to obtain Eq. (\ref{apprAIR_HP}), the results in \cite{CL00} indicate that
\[\mathbb{E}[K_{T}] \approx \frac{Q_{T}+1}{\lambda T} \quad \mbox{and} \quad \mathbb{E}[H_{TP}]\approx TQ_{T}+\frac{Q_{T}(Q_{T}+1)}{2\lambda},\] whereas
\begin{eqnarray}
\mathbb{E}\left[L_{TP}^{C}\right]=T, \quad \mathbb{E}\left[L_{TP}^{R}\right]=M_{G_{T}}(Q_{T})T\approx\frac{Q_{T}+1}{\lambda},\label{Length_TP}
\end{eqnarray}
\begin{eqnarray}
\mathbb{E}[RCost_{TP}]=A_{R}+c_{R}\mathbb{E}[K_{T}]\lambda T, \quad \mbox{and} \quad \mathbb{E}[DCost_{TP}]=A_{D}\mathbb{E}[K_{T}]+c_{D}\mathbb{E}[K_{T}]\lambda T. \label{cost_TP}
\end{eqnarray}
Also, using the definitions of $AOD$ and $AIR$ given by Eqs.~(\ref{aod-1}) and (\ref{air-1}), it is easy to show that
\begin{eqnarray}
AOD_{TP}=T/2, \quad\mbox{and} \quad  AIR_{TP}=\frac{\sum_{i=0}^{Q_{T}}(Q_{T}-i)m_{g_{T}}(i)}{M_{G_{T}}(Q_{T})}\approx\frac{Q_{T}(2\lambda T+Q_{T}+1)}{2(Q_{T}+1)}.\label{AODAIR_TP}
\end{eqnarray}

In the integrated model under QP in \cite{CML06}, $n$ denotes the number of consolidation cycles within an inventory replenishment cycle and $q$ denotes the targeted consolidation size. The corresponding
expected cumulative inventory carrying per replenishment cycle  and the expected
cumulative waiting per consolidation cycle
are given by
\[
\mathbb{E}[H_{QP}]=\frac{1}{2\lambda}n(n-1)q^{2}
\quad \mbox{and} \quad
\mathbb{E}[W_{QP}]=\frac{1}{2\lambda}(q-1)q,
\]
respectively, whereas
\begin{equation}
\mathbb{E}\left[L_{QP}^{C}\right]=\frac{q}{\lambda},  \mathbb{E}\left[L_{QP}^{R}\right]=\frac{nq}{\lambda},
\mathbb{E}[RCost_{QP}]=A_{R}+c_{R}nq, \; \mbox{and} \; \mathbb{E}[DCost_{QP}]=nA_{D}+c_{D}nq. \label{Length-costQP}
\end{equation}
Also, using the definitions of $AOD$ and $AIR$ given by Eqs.~(\ref{aod-1}) and (\ref{air-1}), it is easy to show that
\begin{equation}
AOD_{QP}=\frac{q-1}{2\lambda}, \quad \mbox{and} \quad  AIR_{QP}=(n-1)q/2.\label{AODAIR_QP}
\end{equation}

Recalling Eqs.~\eqref{eq:consolidation cycle}, \eqref{eq:replenishment cost}, \eqref{appr expected repl length}, and \eqref{eq:dispatch cost} developed  in Eqs. \eqref{Length_TP}, \eqref{cost_TP}, and \eqref{Length-costQP} that were previously developed in \cite{CML06} and \cite{CL00}, we summarize the exact and approximate expressions of expected cycle lengths, inbound replenishment and outbound dispatch costs per replenishment
cycle in Table \ref{t1}.

\begin{table}[ht]
\centering{} %
\begin{tabular}{ll}
\hline
$\mathbb{E}\left[L_{QP}^{R}\right]=nq/\lambda$  &   $\mathbb{E}\left[L_{QP}^{C}\right]=q/\lambda$ \tabularnewline\\
$\mathbb{E}\left[L_{TP}^{R}\right]=M_{G_{T}}(Q_{T})T\approx(Q_{T}+1)/\lambda$  &   $\mathbb{E}\left[L_{TP}^{C}\right]=T$ \tabularnewline\\
$\mathbb{E}\left[L_{HP}^{R}\right]=M_{G_{H}}(Q_{H})\mathbb{E}[Y_{q_{H}}]/\lambda \approx (Q_{H}+1)/\lambda$  &   $\mathbb{E}\left[L_{HP}^{C}\right]=\mathbb{E}[Y_{q_{H}}]/\lambda$ \tabularnewline\\
$\mathbb{E}[RCost_{QP}]=A_{R}+c_{R}nq$ &
$\mathbb{E}[DCost_{QP}]=nA_{D}+c_{D}nq$ \tabularnewline\\
$\mathbb{E}[RCost_{TP}]=A_{R}+c_{R}\mathbb{E}[K_{T}]\lambda T$ &
$\mathbb{E}[DCost_{TP}]=A_{D}\mathbb{E}[K_{T}]+c_{D}\mathbb{E}[K_{T}]\lambda T$ \tabularnewline\\
$\mathbb{E}[RCost_{HP}]=A_{R}+c_{R}\mathbb{E}[K_{H}]\mathbb{E}[Y_{q_{H}}]$ &
$\mathbb{E}[DCost_{HP}]=A_{D}\mathbb{E}[K_{H}]+c_{D}\mathbb{E}[K_{H}]\mathbb{E}[Y_{q_{H}}]$ \tabularnewline
\hline
\end{tabular}\caption{Expected cycle lengths and replenishment and dispatch costs per replenishment
cycle.}
\label{t1}
\end{table}

Moreover, recalling Eqs.~\eqref{AIR_HP}, \eqref{apprAIR_HP}, \eqref{AOD_HP}, \eqref{AOSD_HP}, \eqref{AOSD_QP and TP}, \eqref{AODAIR_TP}, and \eqref{AODAIR_QP} developed here, we summarize the exact and approximate expressions of $AOD$, $AOSD$, and $AIR$ in Table \ref{t2}. We use both Tables \ref{t1} and \ref{t2} for the purposes of our comparative analysis in the remainder of the paper.
\begin{table}[ht]
\centering{}\hspace*{-0.5in} %
\begin{tabular}{llll}
\hline
 %&  &  & \tabularnewline
$AOD_{QP}$  & $=$ & {\Large{}$\frac{q-1}{2\lambda}$}  & \tabularnewline
\\
$AOD_{TP}$  & $=$  & $T/2$  & \tabularnewline
\\
$AOD_{HP}$  & $=$  & {\Large{}$\frac{\mathbb{E}\left[Y_{q_{H}}(Y_{q_{H}}-1)\right]/(2\lambda)}{\mathbb{E}\left[Y_{q_{H}}\right]}$}  & \tabularnewline
\\
$AOSD_{QP}$  & $=$  & {\Large{}$\frac{q^{2}-1}{3\lambda^{2}}$}  & \tabularnewline
\\
$AOSD_{TP}$  & $=$  & $T^{2}/3$  & \tabularnewline
\\
$AOSD_{HP}$  & $=$  & {\Large{}$\frac{\mathbb{E}\left[Y_{q_{H}+1}(Y_{q_{H}+1}-1)(Y_{q_{H}+1}-2)\right]/(3\lambda^{2})}{\mathbb{E}\left[Y_{q_{H}}\right]}$} \tabularnewline
\\
$AIR_{QP}$  & $=$  & $(n-1)q/2$\tabularnewline
\\
$AIR_{TP}$  & $=$  &
{\Large{}{}{}
$\frac{\sum_{i=0}^{Q_{T}}(Q_{T}-i)m_{g_{T}}(i)}{M_{G_{T}}(Q_{T})}\approx\frac{Q_{T}(2\lambda T+Q_{T}+1)}{2(Q_{T}+1)}$}\tabularnewline
\\
$AIR_{HP}$  & $=$  &
{\Large{}
$\frac{\sum_{i=0}^{Q_{H}}(Q_{H}-i)m_{g_{H}}(i)}{M_{G_{H}}(Q_{H})}\approx\frac{Q_{H}(2\mathbb{E}\left[Y_{q_{H}}\right]+Q_{H}+1)}{2(Q_{H}+1)}$}\tabularnewline
\hline
\end{tabular}\caption{Expressions of $AOD$, $AOSD$, and $AIR$.}
\label{t2}
\end{table}

\subsection{New and Refined Properties of Truncated Poisson Random Variables}
\label{trun-prop}

The following lemma offers the key technical result which is useful to prove the new and refined properties of truncated Poisson random variables revealed by Lemmas \ref{lem:useful1} and \ref{lem:useful2}.
\begin{lem}
\label{lem:key}Suppose $X\sim\text{Poisson}(\mu)$, $V_{n}\sim\text{gamma}(n,1)$
for integer $n\geq1$, we have
\[
\mathbb{E}\left[X_{q}^{(k)}\right]\triangleq\mathbb{E}\left[X_{q}(X_{q}-1)\cdots(X_{q}-k+1)\right]=\mathbb{E}\left[V_{q-k+1}^{k}\wedge\mu^{k}\right],
\]
\[
\frac{d}{d\mu}\mathbb{E}\left[X_{q}^{(k)}\right]=k\mu^{k-1}\mathbb{P}(X\leq q-k)
\]
where $q$ and $k$ are two positive integers and $k\leq q$.
\end{lem}

The property of truncated Poisson random variables revealed by the following lemma is fundamental for the purpose of comparing HP and TP in terms of $AOD$ under a fixed expected consolidation cycle length.

\begin{lem}
\label{lem:useful1} Suppose $X\sim\text{Poisson}(\mu)$, then $\mathbb{E}^{2}\left[X_{q}\right]/\mathbb{E}\left[X_{q}^{(2)}\right]$
is increasing in $\mu$, and $\mathbb{E}^{2}\left[X_{q}\right]>\mathbb{E}\left[X_{q}^{(2)}\right]$ for all $\mu>0$,
where $\mathbb{E}\left[X_{q}^{(2)}\right]\triangleq\mathbb{E}\left[X_{q}(X_{q}-1)\right]$.
\end{lem}

Before we present another new property of truncated Poisson random variables in Lemma \ref{lem:useful2}, we state an intuitive prerequisite result supporting its proof.
\begin{lem}
\label{lem:var} Suppose $X\sim\text{Poisson}(\mu)$,
and $A\subset\mathbb{Z}_{+}$, then
\[
\frac{d}{d\mu}\mathbb{E}\left[X|X\in A\right]=\frac{1}{\mu}\mathbb{VAR}\left[X|X\in A\right]\geq0,
\]
where the equality holds only if $A$ is a singleton set.
\end{lem}

Last but not least, the following additional property of truncated Poisson random variables is useful when we compare HP and TP in terms of $AOSD$ for a fixed expected consolidation cycle length.

\begin{lem}
\label{lem:useful2} Suppose $X\sim\text{Poisson}(\mu)$, then there
exists some $\mu'>0$, such that $\mathbb{E}^{3}\left[X_{q}\right]/\mathbb{E}\left[X_{q+1}^{(3)}\right]$
is increasing on $(0,\mu')$ and decreasing on $(\mu',\infty)$, and
$\mathbb{E}^{3}\left[X_{q}\right]>\mathbb{E}\left[X_{q+1}^{(3)}\right]$ for all $\mu>0$,
where $\mathbb{E}\left[X_{q+1}^{(3)}\right]\triangleq\mathbb{E}\left[X_{q+1}(X_{q+1}-1)(X_{q+1}-2)\right]$.
\end{lem}

\subsection{Comparative Results}
\label{results}

We now proceed with a comparison of the three integrated models in terms of $AOD$, $AOSD$ and $AIR$ as well as the cost performance.

\subsubsection{Order Delay Penalty Comparison}

We note that, for the purpose of a sensible comparison in terms of service frequency, the comparative results for $AOD$ and $AOSD$ consider a fixed expected consolidation cycle length, and, hence a fixed dispatch/service frequency, among the three shipment consolidation policies.

We begin with our first main result which demonstrates that under a fixed expected consolidation
cycle length, QP performs the best, TP performs the worst in terms
of $AOD$, and HP lies between QP and TP.

\begin{thm}
\label{thm:AOD}Under a fixed expected consolidation cycle length, $AOD_{QP}<AOD_{HP}<AOD_{TP}$. As a result, under a fixed expected consolidation cycle length and a fixed expected replenishment cycle length, we have
$\mathbb{E}[WCost_{QP}]<\mathbb{E}[WCost_{HP}]<\mathbb{E}[WCost_{TP}]$.
\end{thm}

\begin{rem}
\label{rem:aod}
The intuition behind Theorem \ref{thm:AOD} is as follows. Under a QP, each shipment is released immediately after
the last order of the cycle arrives. However, under a HP, random some time elapses between the arrival of the last order and the shipment release epoch. That is, the entire consolidated load of a HP cycle has to wait for this extra time unlike the case under a QP. As a result, for a given expected consolidation cycle length, QP dominates the counterpart HP in terms of $AOD$.
\end{rem}

We proceed with a formal result demonstrating that TP performs worse than QP and HP in terms of $AOSD$, under a fixed
expected consolidation cycle length. However, under a fixed expected consolidation cycle length, QP does not necessarily achieve smaller $AOSD$ than HP, which is different from the corresponding result for $AOD$.

\begin{thm}
\label{thm:AOSD} Under a fixed expected consolidation cycle length,
$
AOSD_{QP}<AOSD_{TP}$ and $AOSD_{HP}<AOSD_{TP}$;
but, neither $AOSD_{QP}$ nor $AOSD_{HP}$ dominates the other one.
\end{thm}

\begin{rem}
\label{rem:QPvsHPofAOSD}
Under a fixed expected consolidation cycle length, the fact that neither $AOSD_{QP}$ nor $AOSD_{HP}$ dominates the other one is different from the corresponding result for $AOD$. The reason is that $AOSD$ assigns a disproportional weight (via the squared delay) to large waiting times. While HP has a maximum limit on the waiting time, QP does not. Hence, depending on the value of this limit imposed by HP relative to the demand volume, QP or HP may be preferable in terms of $AOSD$.
\end{rem}

\subsubsection{Inventory Stagnancy/Flow Comparison}

Next, we proceed with an examination of $AIR$ as a critical measure of interest indicating inventory stagnancy/flow.
For the purpose of a sensible comparison, along with the dispach/service frequency, we also fix the inventory turnover frequency, i.e., we consider fixed values of expected consolidation and replenishment cycle lengths.
Since the exact expressions of $AIR$ we have obtained in this paper do not lend themselves for an exact comparison, the comparison results of Theorem \ref{thm:AIR comparison} hold in an approximate fashion.

%\textcolor{red}{For the sake of a fair comparison, our investigation below assumes fixed values of expected consolidation and replenishment cycle lengths. This assumption is also necessary for a balanced comparison of average costs among the three models. With this approach, as we fix the expected dispatch frequency within one replenishment cycle, we also fix the expected outbound dispatch costs per replenishment cycle along with the expected inbound replenishment costs that provide a basis of the comparison. From the comparison result in terms of $AOD$, we have the comparison result for the expected cumulative waiting cost per replenishment cycle. Thus, in order to compare the expected long run average costs of the integrated models under QP, TP, and HP, it suffices to compare the $AIR$ expressions only. This is the motivation of Theorem \ref{thm:AIR comparison}. Since the exact expressions of $AIR$ we obtained in this paper do not lend themselves for an exact comparison, the comparison results of Theorem \ref{thm:AIR comparison} hold in an approximate fashion. }
\begin{thm}
\label{thm:AIR comparison}
Under a fixed expected
consolidation cycle length and a fixed expected replenishment
cycle length, $AIR_{TP}\approx AIR_{HP}\gtrsim AIR_{QP}$.
As a result, under a fixed expected
consolidation cycle length and a fixed expected replenishment
cycle length,
$\mathbb{E}[HCost_{TP}]\approx\mathbb{E}[HCost_{HP}]\gtrsim \mathbb{E}[HCost_{QP}]$.
\end{thm}

\begin{rem}
\label{rem: variance} Under a fixed expected consolidation cycle length,
the means of the consolidated loads of each outbound dispatch under QP, HP, and TP are the same ($q=\mathbb{E}\left[Y_{q_{H}}\right]=\lambda T$), and the
variances are $0$, $\mathbb{VAR}\left[Y_{q_{H}}\right]$, and $\lambda T$,
respectively. From Lemma \ref{lem:useful1}, $\mathbb{VAR}\left[Y_{q_{H}}\right]<\mathbb{E}\left[Y_{q_{H}}\right]$, i.e.,
under a fixed expected consolidation cycle length, the
variance of dispatch quantity in each outbound under HP lies between those under QP and TP, and this result offers intuition to explain Theorem \ref{thm:AIR comparison}.
\end{rem}

\begin{rem}
Recalling the exact expressions of expected consolidation and replenishment cycle lengths in Table \ref{t1}, we note that, by assumption, Theorem \ref{thm:AIR comparison} requires
\begin{eqnarray}
 \lambda T=\mathbb{E}\left[Y_{q_{H}}\right]=q \quad\mbox{and}\quad M_{G_{T}}(Q_{T})=M_{G_{H}}(Q_{H})=n.\label{exactfixedleng}
\end{eqnarray}
Clearly, given values of expected consolidation and replenishment cycle lengths, say $\mathbb{E}\left[L^{C}\right]$ and $\mathbb{E}\left[L^{R}\right]$, respectively, there may or may not exist a solution
\[
(n,q,Q_T,T,Q_H,q_H, T_H)
\]
that satisfies Eq.~\eqref{exactfixedleng} such that
 $n,q,Q_T,Q_H$ and $q_H$ are all integers.
Our proof  of Theorem \ref{thm:AIR comparison} in Appendix \ref{proofofAIRcomp}, however, sets the approximate expression of the replenishment cycle length in  Table \ref{t1} equal to $\mathbb{E}[L^{R}]$ so that we work with the solution such that
\begin{eqnarray*}
\lambda T=\mathbb{E}\left[Y_{q_{H}}\right]=q\quad\mbox{and}\quad Q_{T}+1\approx Q_{H}+1\approx nq. %\label{fixedtwoexpectedcycles}
\end{eqnarray*}
Clearly, the above equations can always be satisfied by some carefully selected values of $n,q,Q_T,T,Q_H, q_H$, and $T_H$ which may or not be implementable (e.g., due to non-integer values) in the context of the three integrated models (e.g., due to non-integer solutions). Hence, we say that the theorem is applicable in practice whenever such an implementable solution exists.
\end{rem}

\subsubsection{Cost Comparison}

We let $AC_p$ denote the expected long run average cost per time unit for $p=QP, TP, HP$, and we again consider fixed dispatch/service and inventory turnover frequencies, i.e., fixed values of expected consolidation and replenishment cycle lengths. Then, the expected number of outbound dispatches within a replenishment cycle is the same under the three policies. As a result, both the resulting inventory replenishment costs per replenishment cycle and the resulting dispatch costs per replenishment cycle under the three policies are the same. This scenario, in turn, allows us to compare the resulting $AC_p$ values for
$p=QP, TP, HP$ without explicitly solving the three underlying optimization problems.

\begin{thm}
\label{thm:AC_comp} In the linear delay penalty case, under a fixed
expected consolidation cycle length and a fixed replenishment
cycle length,
\[
AC_{QP}\lesssim AC_{HP}\lesssim AC_{TP}.
\]
\end{thm}

Observe that the theorem relies on the approximation result of Eq.~(\ref{appox of EK}) which is obtained by treating $K_{H}$ as a continuous random variable, an assumption which has been used in \cite{CL00,WALD44} as well for obtaining similar approximations. As mentioned in \cite{Axsa01} and demonstrated in \cite{CML06}, the approximation extremely well except in the cases where the distribution warehouse is operated as a transshipment point (i.e., without holding any inventory so that there is only a single consolidation cycle in a replenishment cycle). Hence, the approximation is very effective for all practical purposes in which the warehouse carries inventory. If there is indeed only one consolidation cycle within a replenishment cycle, so that  no inventory is held at the vendor's warehouse, then all of our comparative results are still true. This is because (i) there is no inventory holding cost term to worry about and (ii) the comparison results in terms of the service measures $AOD$ and $AOSD$ still hold, so that the comparison results in terms of the average cost criteria are remain true.

\begin{rem}
In the squared delay penalty case, by a similar argument (using Theorems \ref{thm:AOSD} and \ref{thm:AIR comparison}), we have that under a fixed
expected consolidation length and a fixed replenishment
cycle length,  $AC_{QP}\lesssim AC_{TP}$ and $AC_{HP} \lesssim AC_{TP}$ but neither $AC_{QP}$ nor $AC_{HP}$ beats the other one.

\end{rem}

\section{\label{sec:Conclusion}Conclusion}

In this paper, we have focused on the exact modeling of HPs in an integrated  framework for making inbound replenishment and outbound dispatch decisions simultaneously. The accompanying goal of our modeling effort has been offering an analytical comparison of the resulting model with the previously developed counterpart models that consider TPs and QPs. Considering the $AOD$, $AOSD$ and $AIR$, as well as the cost, as potential evaluation criteria, we have obtained some insightful comparative results in terms of the relative performance of the models of interest. We have also offered some new and refined properties of truncated Poisson random variables.
Our comparative results regarding the cost performance of the models is notable because, unlike the previous work on the topic, it does not require numerical validation or optimization regardless of the input parameters of the underlying models.

Overall, our results demonstrate the relative impact and value of HPs over QPs and TPs, analytically. One fundamental finding in this regards is the analytically provable value of HPs over TPs, in general, as well as its potential value over QPs in terms of $AOSD$, i.e., when dealing with impatient customers. Collectively, these findings are of practical value in the context of  design and operation of an integrated framework for inventory-transportation systems and in selecting an operational policy for shipment consolidation practices in an integrated framework.

Interesting and challenging areas of future research motivated by and extending our results include consideration of more general demand processes (renewal processes or Brownian motion) and investigation of the structure of exact optimal policies relative to QPs, TPs, and HPs.

\bibliographystyle{plain}
\bibliography{References}

\newpage
\setcounter{page}{1}
\pagenumbering{roman}
\appendix

\section{Online Appendix}
\subsection{Squared Delay Penalty \label{sec:Squared-Delay}}

Let $\tau_{n}$ be the first hitting time of $n$ with respect to $N(t)$, where $n$ is a positive integer, so that $\tau_{n}$ is distributed as $\text{gamma}(n,\lambda)$. Let the positive integer $q$ and $T>0$ denote two parameters of the HP under our consideration. Then, the expected cumulative squared waiting per shipment consolidation cycle is
\begin{eqnarray}
\mathbb{E}[W'_{HP}]\!\!\!\!\!
& = & \!\!\! \mathbb{E}\left[\int_{0}^{\tau_{q}\wedge T}(\tau_{q}\wedge T-t)^{2}dN(t)\right]\nonumber \\
\!\!\!\!\!\!& = & \!\!\!\!\!\mathbb{E}\left[\int_{0}^{\tau_{q}\wedge T}\!\!\!(\tau_{q}\wedge T)^{2}dN(t)\right] \!\!- \!\!2\mathbb{E}\left[\int_{0}^{\tau_{q}\wedge T}\!\!t(\tau_{q}\wedge T)dN(t)\right] \!\!+\mathbb{E}\left[\int_{0}^{\tau_{q}\wedge T}t^{2}dN(t)\right].\label{eq:quadratic penalty-1}
\end{eqnarray}
In order to compute the three terms on the right hand side
of Eq. (\ref{eq:quadratic penalty-1}), first, let us use
\begin{eqnarray}
\mathbb{E}[\tau_{q}^{2}1_{\tau_{q}\leq T}]&=&\frac{q(q+1)}{\lambda^{2}}\mathbb{P}(N(T)\geq q+2),
\quad \mbox{so that}
\nonumber\\
\mathbb{E}\left[\int_{0}^{\tau_{q}\wedge T}(\tau_{q}\wedge T)^{2}dN(t)\right] & = & \mathbb{E}\left[(\tau_{q}\wedge T)^{2}N(\tau_{q}\wedge T)\right] \nonumber\\
&=&  q\mathbb{E}\left[\tau_{q}^{2}1_{\tau_{q}\leq T}\right]+T^{2}\mathbb{E}\left[N(T)1_{N(T)\leq q-1}\right]\nonumber \\
 & = & \frac{q^{2}(q+1)}{\lambda^{2}}\mathbb{P}(N(T)\geq q+2)+T^{2}\sum_{n=0}^{q-1}n\mathbb{P}(N(T)=n).\label{first term}
\end{eqnarray}
Second, observe that $2\mathbb{E}\left[\int_{0}^{\tau_{q}\wedge T}t(\tau_{q}\wedge T)dN(t)\right]$ is given by
\begin{eqnarray}
& & \!\!\!\!\!\!\!\!\!\!\!\!\!\!\!\!2\int_{0}^{T}\mathbb{E}\left[\int_{0}^{\tau_{q}}t\tau_{q}dN(t)|\tau_{q}=s\right]f_{\tau_{q}}(s)ds+2T\mathbb{E}\left[\int_{0}^{T}tdN(t)1_{N(T)\leq q-1}\right]\nonumber \\
 & = & 2\int_{0}^{T}s\left(\mathbb{E}\left[\sum_{i=1}^{q-1}\tau_{i}|\tau_{q}=s\right]+s\right)f_{\tau_{q}}(s)ds+2T\sum_{n=0}^{q-1}\mathbb{E}\left[\sum_{i=1}^{n}\tau_{i}|N(T)=n\right]\mathbb{P}(N(T)=n)\nonumber \\
 & = & 2\int_{0}^{T}s((q-1)\frac{s}{2}+s)f_{\tau_{q}}(s)ds+2T\sum_{n=0}^{q-1}n\frac{T}{2}\mathbb{P}(N(T)=n)\nonumber \\
 & = & (q+1)\mathbb{E}\left[\tau_{q}^{2}1_{\tau_{q}\leq T}\right]+T^{2}\sum_{n=0}^{q-1}n\mathbb{P}(N(T)=n)\nonumber \\
 & = & \frac{q(q+1)^{2}}{\lambda^{2}}\mathbb{P}(N(T)\geq q+2)+T^{2}\sum_{n=0}^{q-1}n\mathbb{P}(N(T)=n),\label{second term}
\end{eqnarray}
where $f_{\tau_{q}}(\cdot)$ is the density of $\tau_{q}\sim\text{gamma}(q,\lambda)$,
and the third equality is derived from Lemma 4.5.1 and Theorem 4.5.2
in \cite[p. 322, 325]{RESN02}. Third, noting that $\int_{0}^{t}s^{2}dN(s)-\frac{1}{3}\lambda t^{3}$
is a martingale, and $\tau_{q}\wedge T$ is a bounded stopping time,
we apply the martingale stopping theorem to obtain
\begin{eqnarray}
 &&\mathbb{E}\left[\int_{0}^{\tau_{q}\wedge T}t^{2}dN(t)\right]=  \frac{1}{3}\lambda\mathbb{E}\left[(\tau_{q}\wedge T)^{3}\right]=  \frac{1}{3}\lambda\mathbb{E}\left[\tau_{q}^{3}1_{\tau_{q}\leq T}\right]+\frac{1}{3}\lambda T^{3}\mathbb{P}(N(T)\leq q-1)\nonumber \\
&&\quad \quad =  \frac{(q+2)(q+1)q}{3\lambda^{2}}\mathbb{P}(N(T)\geq q+3)+\frac{1}{3\lambda^{2}}\sum_{m=0}^{q+2}m(m-1)(m-2)\mathbb{P}(N(T)=m)\nonumber \\
&&\quad \quad = \frac{1}{3\lambda^{2}}\mathbb{E}\left[Y_{q+2}(Y_{q+2}-1)(Y_{q+2}-2)\right],\label{thirdterm}
\end{eqnarray}
where we use
\[
\mathbb{E}\left[\tau_{q}^{3}1_{\tau_{q}\leq T}\right]=\frac{(q+2)(q+1)q}{\lambda^{3}}\mathbb{P}(N(T)\geq q+3),
\]
and
\[
\lambda T^{3}\mathbb{P}(N(T)=n)=\frac{(n+3)(n+2)(n+1)}{\lambda^{2}}\mathbb{P}(N(T)=n+3).
\]

Substituting Eqs. (\ref{first term}), (\ref{second term}) and (\ref{thirdterm})
in Eq. (\ref{eq:quadratic penalty-1}), we obtain
\begin{eqnarray*}
\mathbb{E}\left[W'_{HP}\right] & = & \frac{(q+1)q(q-1)}{3\lambda^{2}}\mathbb{P}(N(T)\geq q+2)+\frac{1}{3\lambda^{2}}\sum_{m=0}^{q+1}m(m-1)(m-2)\mathbb{P}(N(T)=m)\nonumber \\
 & = & \frac{1}{3\lambda^{2}}\mathbb{E}\left[Y_{q+1}(Y_{q+1}-1)(Y_{q+1}-2)\right].
\end{eqnarray*}

\subsection{Proofs\label{sec:Proofs}}

\subsubsection{Proof of Lemma \ref{lem:martin}}
\begin{proof}
Since the Poisson process has stationary independent increments, for $s<t$, we have
\begin{eqnarray*}
\mathbb{E}\Bigl[\int_{0}^{t}N(u)du\mid\mathcal{G}_{s}\Bigr] & = & \int_{0}^{s}N(u)du+\mathbb{E}\Bigl[\int_{s}^{t}N(u)du\mid\mathcal{G}_{s}\Bigr]\\
 & = & \int_{0}^{s}N(u)du+(t-s)N(s)+\mathbb{E}\Bigl[\int_{0}^{t-s}N(u)du\Bigr]\\
 & = & \int_{0}^{s}N(u)du+(t-s)N(s)+\frac{1}{2}\lambda(t-s)^{2},
\\
\frac{1}{2\lambda}\mathbb{E}[N^{2}(t)\mid\mathcal{G}_{s}]&=&\frac{1}{2\lambda}\Bigl(N^{2}(s)+2\lambda(t-s)N(s)+\lambda(t-s)+\lambda^{2}(t-s)^{2}\Bigr), \quad \mbox{and}
\\
\frac{1}{2\lambda}\mathbb{E}[N(t)\mid\mathcal{G}_{s}]&=&\frac{1}{2\lambda}\left(N(s)+\lambda(t-s)\right).
\end{eqnarray*}
Therefore,
\begin{eqnarray*}
 \mathbb{E}\Bigl[\int_{0}^{t}N(u)du-\frac{1}{2\lambda}N^{2}(t)+\frac{1}{2\lambda}N(t)\mid\mathcal{G}_{s}\Bigr]
 & = & \int_{0}^{s}N(u)du-\frac{1}{2\lambda}N^{2}(s)+\frac{1}{2\lambda}N(s),
\end{eqnarray*}
which shows that $W(t)-\frac{1}{2\lambda}N^{2}(t)+\frac{1}{2\lambda}N(t)$
is a martingale.
\end{proof}

\subsubsection{Proof of Lemma \ref{lem:key}}
\begin{proof}
Using the relationship between Poisson and gamma distribution,
\begin{eqnarray*}
\mathbb{E}\left[X_{q}^{(k)}\right] & = & \sum_{x=0}^{q}x^{(k)}\frac{e^{-\mu}\mu^{x}}{x!}+\sum_{x=q+1}^{\infty}q^{(k)}\frac{e^{-\mu}\mu^{x}}{x!}
 =  \mu^{k}\mathbb{P}(X\leq q-k)+q^{(k)}\mathbb{P}(X\geq q+1)\\
 & = & \mu^{k}\mathbb{P}(V_{q-k+1}>\mu)+q^{(k)}\mathbb{P}(V_{q+1}\leq\mu)
 =  \int_{\mu}^{\infty}\mu^{k}\frac{v^{q-k}e^{-v}}{(q-k)!}dv+\int_{0}^{\mu}v^{k}\frac{v^{q-k}e^{-v}}{(q-k)!}dv\\
 & = & \mathbb{E}\left[V_{q-k+1}^{k}\wedge\mu^{k}\right].
\end{eqnarray*}
Also, recalling the following two properties of Poisson random variable,
\[
\mu^{k}\mathbb{P}(X=q-k)=q^{(k)}\mathbb{P}(X=q) \quad \mbox{and} \quad
\frac{d}{d\mu}\mathbb{P}(X\leq q)=-\mathbb{P}(X=q),
\]
it is straightforward to show that
\begin{eqnarray*}
\frac{d}{d\mu}\mathbb{E}\left[X_{q}^{(k)}\right]= k\mu^{k-1}\mathbb{P}(X\leq q-k)-\mu^{k}\mathbb{P}(X=q-k)+q^{(k)}\mathbb{P}(X=q)
 =  k\mu^{k-1}\mathbb{P}(X\leq q-k).
\end{eqnarray*}
\end{proof}

\subsubsection{Proof of Lemma \ref{lem:useful1}}
\begin{proof}
Using Lemma \ref{lem:key}, it is straightforward to show that
\begin{eqnarray}
\frac{d}{d\mu}\frac{\mathbb{E}^{2}\left[X_{q}\right]}{\mathbb{E}\left[X_{q}^{(2)}\right]}=\frac{2\mathbb{E}\left[X_{q}\right]}{\mathbb{E}^{2}\left[X_{q}^{(2)}\right]}\left(\mathbb{P}(X\leq q-1)\mathbb{E}\left[X_{q}^{(2)}\right]-\mu\mathbb{E}\left[X_{q}\right]\mathbb{P}(X\leq q-2)\right).\label{eq:main1}
\end{eqnarray}
Noting that $\mu^{k}\mathbb{P}(X=q-k)=q^{(k)}\mathbb{P}(X=q)$, we have
\begin{equation}
\mathbb{E}\left[X_{q}^{(2)}\right]=q^{(2)}\mathbb{P}(X\geq q+1)+\sum_{n=0}^{q}n^{(2)}\mathbb{P}(X=n)=q^{(2)}\mathbb{P}(X\geq q+1)+\mu^{2}\mathbb{P}(X\leq q-2),\label{eq:cross1}
\end{equation}
and
\begin{equation}
\mathbb{E}\left[X_{q}\right]=q\mathbb{P}(X\geq q+1)+\sum_{n=0}^{q}n\mathbb{P}(X=n)=q\mathbb{P}(X\geq q+1)+\mu\mathbb{P}(X\leq q-1).\label{eq:cross2}
\end{equation}
Substituting Eqs.~(\ref{eq:cross1}) and (\ref{eq:cross2}) in Eq.
(\ref{eq:main1}), and using $\mu\mathbb{P}(X=n-1)=n\mathbb{P}(X=n)$,
we have
\begin{align*}
\frac{d}{d\mu}\frac{\mathbb{E}^{2}\left[X_{q}\right]}{\mathbb{E}\left[X_{q}^{(2)}\right]} & =\frac{2q\mathbb{E}\left[X_{q}\right]\mathbb{P}(X\geq q+1)}{\mathbb{E}^{2}\left[X_{q}^{(2)}\right]}\left((q-1)\mathbb{P}(X\leq q-1)-\mu\mathbb{P}(X\leq q-2)\right)\\
 & =\frac{2q\mathbb{E}\left[X_{q}\right]\mathbb{P}(X\geq q+1)}{\mathbb{E}^{2}\left[X_{q}^{(2)}\right]}\sum_{n=0}^{q-1}(q-1-n)\mathbb{P}(X=n)>0,
\end{align*}
which implies that $\frac{\mathbb{E}^{2}\left[X_{q}\right]}{\mathbb{E}\left[X_{q}^{(2)}\right]}$
is increasing in $\mu$. Moreover, using Lemma \ref{lem:key},
\[
\lim_{\mu\downarrow0}\frac{\mathbb{E}^{2}\left[X_{q}\right]}{\mathbb{E}\left[X_{q}^{(2)}\right]}=\lim_{\mu\downarrow0}\frac{\mathbb{E}^{2}\left[(V_{q}/\mu)\wedge1\right]}{\mathbb{E}\left[(V_{q-1}/\mu)^{2}\wedge1\right]}=1.
\]
Thus, $\mathbb{E}^{2}\left[X_{q}\right]>\mathbb{E}\left[X_{q}^{(2)}\right]$.
\end{proof}

\subsubsection{Proof of Lemma \ref{lem:var}}
\begin{proof}
It can be easily verified that
\begin{eqnarray*}
\frac{d}{d\mu}\mathbb{E}\left[X|X\in A\right] & = & \frac{d}{d\mu}\frac{\sum_{x\in A}x\frac{\mu^{x}}{x!}}{\sum_{x\in A}\frac{\mu^{x}}{x!}}
 =  \frac{\sum_{x\in A}x^{2}\frac{\mu^{x-1}}{x!}\sum_{x\in A}\frac{\mu^{x}}{x!}-\sum_{x\in A}x\frac{\mu^{x}}{x!}\sum_{x\in A}x\frac{\mu^{x-1}}{x!}}{(\sum_{x\in A}\frac{\mu^{x}}{x!})^{2}}\\
 & = & \frac{\sum_{x\in A}x^{2}\frac{\mu^{x}}{x!}}{\mu\sum_{x\in A}\frac{\mu^{x}}{x!}}-\frac{(\sum_{x\in A}x\frac{\mu^{x}}{x!})^{2}}{\mu(\sum_{x\in A}\frac{\mu^{x}}{x!})^{2}}
 =  \frac{1}{\mu}\mathbb{VAR}\left[X|X\in A\right]\geq0.
\end{eqnarray*}
\end{proof}

\subsubsection{Proof of Lemma \ref{lem:useful2}}
\begin{proof}
Using Lemma \ref{lem:key}, it is straightforward to obtain
\begin{equation}
\frac{d}{d\mu}\frac{\mathbb{E}^{3}\left[X_{q}\right]}{\mathbb{E}\left[X_{q+1}^{(3)}\right]}=\frac{3\mathbb{E}^{2}\left[X_{q}\right]}{\mathbb{E}^{2}\left[X_{q+1}^{(3)}\right]}\left(\mathbb{P}(X\leq q-1)\mathbb{E}\left[X_{q+1}^{(3)}\right]-\mu^{2}\mathbb{E}\left[X_{q}\right]\mathbb{P}(X\leq q-2)\right).\label{eq:main2}
\end{equation}
Noting that $\mu^{k}\mathbb{P}(X=q-k)=q^{(k)}\mathbb{P}(X=q)$, we
have
\begin{equation}
\mathbb{E}\left[X_{q+1}^{(3)}\right]=(q^{2}-1)\mathbb{P}(X\geq q+2)+\sum_{n=0}^{q+1}n^{(3)}\mathbb{P}(X=n)=(q^{2}-1)\mathbb{P}(X\geq q+2)+\mu^{3}\mathbb{P}(X\leq q-2),\label{eq:cross3}
\end{equation}
and
\begin{equation}
\mathbb{E}\left[X_{q}\right]=q\mathbb{P}(X\geq q+1)+\sum_{n=0}^{q}n\mathbb{P}(X=n)=q\mathbb{P}(X\geq q+1)+\mu\mathbb{P}(X\leq q-1).\label{eq:cross4}
\end{equation}
Substituting Eqs.~(\ref{eq:cross3}) and (\ref{eq:cross4}) in Eq.~(\ref{eq:main2}), we have
\begin{align}
 & \frac{d}{d\mu}\frac{\mathbb{E}^{3}\left[X_{q}\right]}{\mathbb{E}\left[X_{q+1}^{(3)}\right]}\nonumber \\
= & \frac{3\mu q\mathbb{E}^{2}\left[X_{q}\right]\mathbb{P}(X\geq q+2)\mathbb{P}(X\leq q-2)}{\mathbb{E}^{2}\left[X_{q+1}^{(3)}\right]}\left(\frac{(q^{2}-1)\mathbb{P}(X\leq q-1)}{\mu\mathbb{P}(X\leq q-2)}-\frac{\mu\mathbb{P}(X\geq q+1)}{\mathbb{P}(X\geq q+2)}\right).\label{eq: main3}
\end{align}
From Lemma \ref{lem:var}, we note that
\[
\frac{\mu\mathbb{P}(X\leq q-2)}{\mathbb{P}(X\leq q-1)}=\frac{\mu\sum_{x=0}^{q-2}\frac{\mu^{x}}{x!}}{\sum_{x=0}^{q-1}\frac{\mu^{x}}{x!}}=\frac{\sum_{x=0}^{q-1}x\frac{\mu^{x}}{x!}}{\sum_{x=0}^{q-1}\frac{\mu^{x}}{x!}}=\mathbb{E}\left[X|X\leq q-1\right]
\]
is increasing in $\mu$ (from 0 to $q-1$). Likewise,
\[
\frac{\mu\mathbb{P}(X\geq q+1)}{\mathbb{P}(X\geq q+2)}=\mathbb{E}\left[X|X\geq q+2\right]
\]
is increasing in $\mu$ (from $q+2$ to $\infty$). Therefore, the
last expression in parenthesis in Eq. (\ref{eq: main3}) is decreasing
in $\mu$, positive (and unbounded) for small $\mu$ and negative
(and unbounded) for large $\mu$, which implies that there exists
some $\mu'$, $\mathbb{E}^{3}\left[X_{q}\right]/\mathbb{E}\left[X_{q+1}^{(3)}\right]$
is increasing on $(0,\mu')$ and decreasing on $(\mu',\infty)$. Moreover,
using Lemma \ref{lem:key},
\[
\lim_{\mu\downarrow0}\frac{\mathbb{E}^{3}\left[X_{q}\right]}{\mathbb{E}\left[X_{q+1}^{(3)}\right]}=\lim_{\mu\downarrow0}\frac{\mathbb{E}^{3}\left[(V_{q}/\mu)\wedge1\right]}{\mathbb{E}\left[(V_{q-1}/\mu)^{3}\wedge1\right]}=1,
\]
and
\[
\lim_{\mu\rightarrow\infty}\frac{\mathbb{E}^{3}\left[X_{q}\right]}{\mathbb{E}\left[X_{q+1}^{(3)}\right]}=\frac{q^{3}}{(q+1)^{(3)}}=\frac{q^{2}}{q^{2}-1}>1.
\]
Thus, $\mathbb{E}^{3}\left[X_{q}\right]>\mathbb{E}\left[X_{q+1}^{(3)}\right]$ for all
$\mu>0$.
\end{proof}

\subsubsection{Proof of Theorem \ref{thm:AOD}}
\begin{proof}
The proof follows from the proof of the counterpart result in \cite{WCC20} which considers pure consolidation policies only; but, it is given here for the sake of completeness. We consider a fixed $\mathbb{E}\left[L^{C}\right]$ and use the following notations
for the corresponding policy parameters under this $\mathbb{E}\left[L^{C}\right]$
value: QP with parameter $q$, TP with parameter $T$, HP with parameters
$q_{H}$ and $T_{H}$. Recalling the $\mathbb{E}\left[L^{C}\right]$ expressions
in Table \ref{t1}, we note that, by assumption,
\begin{equation}
\frac{1}{\lambda}\mathbb{E}\left[Y_{q_{H}}\right]=\frac{q}{\lambda},\label{Q-HP-L}
\end{equation}
and
\begin{equation}
\frac{1}{\lambda}\mathbb{E}\left[Y_{q_{H}}\right]=T.\label{HP-T-L}
\end{equation}
Next, recalling the results in Table \ref{t2} and reiterating the
assumption of fixed $\mathbb{E}\left[L^{C}\right]$ values for all the policies
of interest, we proceed with showing that
$
(q-1)q<\mathbb{E}\left[Y_{q_{H}}(Y_{q_{H}}-1)\right],
$
and
$\mathbb{E}\left[Y_{q_{H}}(Y_{q_{H}}-1)\right]<\lambda^{2}T^{2}.
$
In fact, recalling the assumption in Eq.~(\ref{Q-HP-L}),
\[
\mathbb{E}\left[Y_{q_{H}}(Y_{q_{H}}-1)\right]=\mathbb{E}\left[Y_{q_{H}}^{2}\right]-q=\mathbb{VAR}\left[Y_{q_{H}}\right]+\mathbb{E}^{2}\left[Y_{q_{H}}\right]-q>q^{2}-q.
\]
From Lemma \ref{lem:useful1}, and recalling the assumption in Eq.~(\ref{HP-T-L}), we have
$
\mathbb{E}\left[Y_{q_{H}}(Y_{q_{H}}-1)\right]<\mathbb{E}^{2}\left[Y_{q_{H}}\right]=\lambda^{2}T^{2}.
$

The proof of the second part of the theorem is straightforward, and, hence, it is omitted.
\end{proof}

\subsubsection{Proof of Theorem \ref{thm:AOSD}}

\begin{proof}
We consider a fixed $\mathbb{E}\left[L^{C}\right]$ and use the following notation
for the corresponding policy parameters under this $\mathbb{E}\left[L^{C}\right]$
value: QP with parameter $q$, TP with parameter $T$, HP with parameters
$q_{H}$ and $T_{H}$. Recalling the $\mathbb{E}\left[L^{C}\right]$ expressions
in Table \ref{t1}, we note that, by assumption,
\begin{equation}
\frac{q}{\lambda}=T,\label{Q-T-L}
\end{equation}
and
\begin{equation}
\frac{\mathbb{E}\left[Y_{q_{H}}\right]}{\lambda}=T.\label{HP-T-L-2}
\end{equation}

Next, recalling the results in Table \ref{t2} and reiterating the
assumption of fixed $\mathbb{E}\left[L^{C}\right]$ value for all the policies
of interest, we proceed with showing that
\begin{equation}
\frac{q^{3}-q}{3\lambda^{2}}<\lambda T^{3}/3,\label{Q-T-QW}
\end{equation}
and
\begin{equation}
\frac{\mathbb{E}\left[Y_{q_{H}+1}^{(3)}\right]}{3\lambda^{2}}<\lambda T^{3}/3.\label{HP-T-QW}
\end{equation}
Recalling the assumption in Eq.~(\ref{Q-T-L}), we can easily see
that inequality (\ref{Q-T-QW}) holds. From Lemma \ref{lem:useful2},
and recalling the assumption in Eq. (\ref{HP-T-L-2}), we have
$
\mathbb{E}\left[Y_{q_{H}+1}^{(3)}\right]<\mathbb{E}^{3}\left[Y_{q_{H}}\right]=(\lambda T)^{3},
$
which verifies inequality (\ref{HP-T-QW}).

We need to demonstrate that under a fixed expected consolidation cycle
length $\mathbb{E}\left[L^{C}\right]$, neither $AOSD_{QP}$ nor $AOSD_{HP}$ dominates the
other one. To see this, letting $q=\mathbb{E}\left[Y_{q_{H}}\right]$, we only need to compare $q^{3}-q$
with $\mathbb{E}\left[Y_{q_{H}+1}^{(3)}\right]$. The numerical tests demonstrate
that $q^{3}-q>\mathbb{E}\left[Y_{q_{H}+1}^{(3)}\right]$ when $q_{H}$ is small,
but $q^{3}-q<\mathbb{E}\left[Y_{q_{H}+1}^{(3)}\right]$ when $q_{H}$ is large. For example, let the arrival process $N(t)$ be a Poisson process
with rate $1.$ We select $q=5$ as the parameter under QP, $q_{H}=6$
and $T_{H}=5.9199$ as the parameters under HP. We can numerically
verify that QP and HP have the same expected consolidation cycle
length, i.e., $q=\mathbb{E}\left[Y_{q_{H}}\right]$. In this case, $q^{3}-q=120$,
and $\mathbb{E}\left[Y_{q_{H}+1}^{(3)}\right]=\mathbb{E}\left[Y_{q_{H}+1}(Y_{q_{H}+1}-1)(Y_{q_{H}+1}-2)\right]=112.8573$. In this example, we demonstrate that under a fixed expected consolidation cycle
length, QP does not necessarily achieve smaller $AOSD$ than HP, which is different from the $AOD$ case.
\end{proof}

\subsubsection{Proof of Theorem \ref{thm:AIR comparison}}\label{proofofAIRcomp}
\begin{proof}
We consider fixed $\mathbb{E}\left[L^{C}\right]$ and $\mathbb{E}\left[L^{R}\right]$,
and all possible policies under the $\mathbb{E}\left[L^{C}\right]$ and $\mathbb{E}\left[L^{R}\right]$
values. Recalling the $\mathbb{E}\left[L^{C}\right]$ and $\mathbb{E}\left[L^{R}\right]$
expressions in Table \ref{t1}, we note that, by assumption,
\begin{eqnarray}
\lambda T=\mathbb{E}\left[Y_{q_{H}}\right]=q,\qquad Q_{T}+1\approx Q_{H}+1\approx nq. \label{fixedtwoexpectedcycles}
\end{eqnarray}

Next, recalling the $AIR$ expressions in Table \ref{t2}, it's sufficient to show
\begin{eqnarray}
\frac{Q_{T}(2\lambda T+Q_{T}+1)}{2(Q_{T}+1)}\approx \frac{Q_{H}(2\mathbb{E}\left[Y_{q_{H}}\right]+Q_{H}+1)}{2(Q_{H}+1)}, \label{TPvsHP_AIR}
\end{eqnarray}
and
\begin{eqnarray}
\frac{Q_{H}(2\mathbb{E}\left[Y_{q_{H}}\right]+Q_{H}+1)}{2(Q_{H}+1)}\gtrsim \frac{(n-1)q}{2}. \label{HPvsQP_AIR}
\end{eqnarray}

Obviously, Eq.~\eqref{TPvsHP_AIR} follows from Eq.~\eqref{fixedtwoexpectedcycles}. Moreover, using Eq.~\eqref{fixedtwoexpectedcycles}, we have
\begin{align*}
\frac{Q_{H}(2\mathbb{E}\left[Y_{q_{H}}\right]+Q_{H}+1)}{2(Q_{H}+1)}-\frac{(n-1)q}{2} \approx &\frac{(nq-1)(n+2)}{2n}-\frac{(n-1)q}{2}\\
= & \frac{2(nq-1)+n(q-1)}{2n}>0,
\end{align*}
which implies Eq.~\eqref{HPvsQP_AIR} holds. In sum, we have $AIR_{TP}\approx AIR_{HP}\gtrsim AIR_{QP}$.

The proof of the second part of the theorem is straightforward, and, hence, it is omitted.
\end{proof}

\subsubsection{Proof of Theorem \ref{thm:AC_comp}}
\begin{proof}
Considering a fixed replenishment cycle length $\mathbb{E}[L^{R}]$ and a fixed expected consolidation cycle length $\mathbb{E}[L^{C}]$, and recalling the results in Table \ref{t1}, we have
\[
\mathbb{E}[K_{T}]=\mathbb{E}[K_{H}]=n,
\]
\[
\mathbb{E}[RCost_{QP}]=\mathbb{E}[RCost_{HP}]=\mathbb{E}[RCost_{TP}], \quad \mbox{and}
\]
\[
\mathbb{E}[DCost_{QP}]=\mathbb{E}[DCost_{HP}]=\mathbb{E}[DCost_{TP}].
\]
Recalling the average cost rate
\[
AC_p=\frac{\mathbb{E}[HCost_p]+\mathbb{E}[WCost_p]+\mathbb{E}[RCost_p]+\mathbb{E}[DCost_p]}{\mathbb{E}[L^{R}]}, \quad p=QP, TP, HP,
\]
along with Theorems \ref{thm:AOD} and \ref{thm:AIR comparison},
it can be easily shown that $AC_{QP}\lesssim AC_{HP}\lesssim AC_{TP}$.
\end{proof}

\end{document}